\documentclass{amsart}
\usepackage{amsmath,amsfonts,amsthm,amssymb,amsthm,graphicx,mathtools}

\usepackage[mathscr]{eucal}
\usepackage{graphics}
\usepackage{color}
\usepackage{epsfig}
\usepackage{amsmath, graphicx,  amsthm, amsfonts, amscd}
\usepackage{stackrel}
\newcommand\vertarrowbox[3][6ex]{%
  \begin{array}[t]{@{}c@{}} #2 \\
  \left\downarrow\vcenter{\hrule height #1}\right.\kern-\nulldelimiterspace\\
  \makebox[0pt]{\scriptsize#3}
  \end{array}%
}

\def\moc{m}
\def\TT{L}
\def\fplambda{{\mathcal X}}
\def\numb{I}
\def\neweta{\pi}
\def\fsp{\psi}
\def\neweta{\Upsilon}

\def\ttpsi{\tilde{f}}
\def\tpsi{f}
 
\def\f1{{\bf 1}}
\def\zerof{{\bf 0}}
\def\newdel{\theta} 

\DeclareMathOperator*{\essinf}{ess\,inf}
\DeclareMathOperator*{\esssup}{ess\,sup}

\def\theequation{\arabic{section}.\arabic{equation}}
\def\thetheorem{\arabic{section}.\arabic{theorem}}

\newtheorem{theorem}{Theorem}[section]
\newtheorem{lemma}[theorem]{Lemma}
\newtheorem{corollary}[theorem]{Corollary}
\newtheorem{definition}[theorem]{Definition}

\newtheorem{proposition}[theorem]{Proposition}

\newtheorem{assumption}[theorem]{Assumption}

\newtheorem{remark}[theorem]{Remark}

\def\thelemma{\arabic{section}.\arabic{lemma}}
\def\thetheorem{\arabic{section}.\arabic{theorem}}
\def\thecorollary{\arabic{section}.\arabic{corollary}}
\def\thedefinition{\arabic{section}.\arabic{definition}}
\def\theexample{\arabic{section}.\arabic{example}}
\def\theproposition{\arabic{section}.\arabic{proposition}}

\def\theassumption{\arabic{section}.\arabic{assumption}}

\def\theremark{\arabic{section}.\arabic{remark}}

\def\theequation{\arabic{section}.\arabic{equation}}

\newcommand{\la}{\lambda}
\newcommand{\eps}{\varepsilon}
\newcommand{\ph}{\varphi}

\newcommand{\N}{{\mathbb N}}

\newcommand{\R}{{\mathbb R}}

\newcommand{\Z}{{\mathbb Z}}

\newcommand{\PP}{{\mathbb P}}

\newcommand{\calB}{{\mathcal B}}

\newcommand{\calD}{{\mathcal D}}

\newcommand{\calM}{{\mathcal M}}

\newcommand{\calT}{{\mathcal T}}

\newcommand{\lan}{\langle}
\newcommand{\ran}{\rangle}

\newcommand{\supp}{{\rm supp}}

\newcommand{\w}{\wedge}

\newcommand{\pl}{\partial}

\newcommand{\mean}[1]{\langle#1\rangle}

\newcommand{\iy}{\infty}

\newcommand{\be}{\begin{equation}}
	\newcommand{\ee}{\end{equation}}

\newcommand{\noi}{\noindent}
\newcommand{\ds}{\displaystyle}

\newcommand{\manualnames}[1]{
\def\theequation{#1.\arabic{equation}}
\def\thelemma{#1.\arabic{lemma}}
\def\thetheorem{#1.\arabic{theorem}}
\def\thecorollary{#1.\arabic{corollary}}
\def\thedefinition{#1.\arabic{definition}}
\def\theexample{#1.\arabic{example}}
\def\theproposition{#1.\arabic{proposition}}
\def\theassumption{#1.\arabic{assumption}}
\def\theremark{#1.\arabic{remark}}
}

\newcommand{\beginsec}{
\setcounter{equation}{0}
}

\def\fmeass{{\nu}_*}
\def\frenegs{\eta_*}
\def\fq{Q}
\def\fe{E}
\def\fr{R}
\def\fs{S}
\def\fx{X}
\def\fd{D}
\def\fk{K}
\def\fb{B}

\def\fmeasn{\nu^{(N)}}
\def\frenegn{{\eta}^{(N)}}
\def\fmeas{\nu}
\def\freneg{\eta}
\def\fmeass{\nu_*}

\def\fntf{F^{\eta_t}}
\def\fnsf{F^{\eta_s}}

\newcommand{\newspace}{{\mathcal S}_0}
\newcommand\newcspace{{\mathfrak{S}}}

\newcommand{\nrm}[1]{\left\Vert #1 \right\Vert}

\def\newf{\varphi}
\def\newft{\varphi(\cdot,t)}

\def\newfs{\varphi(\cdot,s)}
\def\dxnewfs{\newf_x(\cdot,s)}
\def\dtnewfs{\newf_s(\cdot,s)}

\def\incspace{{\mathcal I}_{\R_+}[0,\infty)}

\def\ecb{{\mathcal C}_b(E)}
\def\ecc{{\mathcal C}_c (E)}
\def\eocd{{\mathcal C}^1(E)}
\def\eocdc{{\mathcal C}^1_c(E)}
\def\eocdb{{\mathcal C}^1_b(E)}

\def\lloc{{\mathcal L}^1_{loc}}

\def\newocdcpms{{\mathcal C}^1_c([0,H^s)\times\R_+)}
\def\newocdcpmr{{\mathcal C}^1_c([0,H^r)\times\R_+)}

\def\mm{{\mathcal M}[0,H)}

\def\mmfs{{\mathcal M}_F[0,H^s)}

  \def\mmfr{{\mathcal M}_F[0,H^r)}
    \def\mmfcr{{\mathcal M}_F^c[0,H^r)}

\def\eem{{\mathcal M}(E)}
\def\eemf{{\mathcal M}_F(E)}

\newcommand{\Ra}{\Rightarrow}
\newcommand{\ra}{\rightarrow}

\newcounter{bean}
\newcommand{\benuma}{\setlength{\labelwidth}{.25in}
\begin{list}%
{(\alph{bean})}{\usecounter{bean}}}
\newcommand{\eenuma}{\end{list}}

\begin{document}

\title[Asymptotics of coupled measure-valued equations]{Large-time Limit of Nonlinearly Coupled
  Measure-valued Equations that Model Many-Server Queues with Reneging}

\author{Rami Atar}
\address{Viterbi Faculty of Electrical and Computer Engineering \\
Technion \\
Haifa,
Israel} 
\email{rami@technion.ac.il}

\author{Weining Kang}
\address{Department of Mathematics and Statistics, \\
  University of Maryland, Baltimore County}
\email{wkang@umbc.edu} 

\author{Haya Kaspi}
\address{Department of Industrial Engineering and Management\\
Technion \\
Haifa,
Israel} 
 \email{hkaspi@technion.ac.il}

\author{  Kavita Ramanan }
\address{Division of Applied Mathematics \\
Brown University \\
Providence, RI 02118, USA }
\email{kavita\_ramanan@brown.edu }
\thanks{KR is  supported
by the U.S. Army Research Office  via Grant W911NF2010133 
and RA is supported by the Israel Science Foundation grant 1035/20.}

\subjclass[2010]{Primary: 60F17, 60K25, 90B22;  Secondary: 60H99, 35D99.}
\keywords{many-server queues, GI/G/N+G queue, fluid limits, reneging,
  abandonment,  measure-valued processes, renewal equation, large-time behavior, stationary distribution, 
  call centers, enzymatic processing networks, transport equation, age-structured population models}

\date{July 12, 2021}

\begin{abstract}
The large-time behavior of a nonlinearly coupled pair of measure-valued transport equations with discontinuous boundary conditions, parameterized by a positive real-valued parameter $\lambda$, is considered. These equations describe the hydrodynamic or fluid limit of many-server queues with reneging (with traffic intensity $\lambda$), which model phenomena in diverse disciplines, including biology and operations research. For a broad class of reneging distributions with finite mean and service distributions with finite mean and hazard rate function that is either decreasing or bounded away from zero and infinity, it is shown that if the fluid equations have a unique invariant state, then the Dirac measure at this state is the unique random fixed point of the fluid equations, which implies that the stationary distributions of scaled $N$-server systems converge to the unique invariant state of the corresponding fluid equations.  Moreover, when $\lambda\ne1$, it is shown that the solution to the fluid equation starting from any initial condition converges to this unique invariant state in the large time limit. The proof techniques are different under the two sets of assumptions on the service distribution. When the hazard rate function is decreasing, a reformulation of the dynamics in terms of a certain renewal equation is used, in conjunction with recursive asymptotic estimates. When the hazard rate function is bounded away from zero and infinity, the proof uses an extended relative entropy functional as a Lyapunov function. Analogous large-time convergence results are also established for a system of coupled measure-valued equations modeling a multiclass queue.  
\end{abstract}

\maketitle

\section{Introduction}

\subsection{Background, Motivation and Results}

The focus of this work is the analysis of  the large-time  behavior of a nonlinearly coupled pair of 
measure-valued  transport equations with discontinuous boundary conditions that describe the hydrodynamic or fluid limit of a many-server
queue with reneging.  
Many-server queues with reneging arise in a range of 
applications, including as models of computer  networks, 
telephone call centers
or (more general) customer contact centers \cite{DaiHe12,ManZel09,PuhWar19},  and 
enzymatic processing networks in biology, where  reneging seeks to model the phenomenon of dilution  (see, e.g., \cite{Matetal10}). 
 A basic model,  also referred to  as the GI/G/N$+$G queue, consists of a system with $N$ identical servers, to which  jobs arrive with independent and 
 identically distributed (i.i.d.) service requirements 
 that are drawn from a general distribution, with each job  also being equipped with
 an i.i.d.\ patience time  drawn from another general distribution.
 Depending on the application, the servers represent  processors, call center agents or enzymes, and
 the jobs represent packets, customers with tasks or proteins. 
 Arriving jobs enter service
immediately if there is an idle server available, else they
join the back  of the queue. As servers become available, jobs 
from the queue start service in the order of arrival.  Once a job completes service, it departs the system.
 In addition, jobs  also renege from the  queue   if the amount of time
they have been in queue exceeds their patience time.  Important system performance measures
of  interest include the stationary waiting time and queue distributions.
In the special case when arrivals are  Poisson and the service distribution is exponential,
but the abandonment distribution is general, 
explicit formulas for the scaled steady-state distributions were obtained in \cite{BacHeb81},
and their asymptotics as $N$, the number of servers,  goes to infinity,
were studied in \cite{ZelMan05}.
However, the case of general service distributions, which is relevant for many applications,
 is  more challenging and 
 it appears not feasible to obtain
 exact  analytical expressions for these quantities for general service and abandonment distributions.
 Instead, one often  resorts to
obtaining asymptotic approximations that are exact in the limit  as the number
of servers goes to infinity. 

In \cite{KanRam10} the state  of an $N$-server queue  at time $t$ 
is represented in terms of  two coupled measures, the queue measure and the server measure.
The queue measure encodes jobs currently in the queue and 
has a unit Dirac delta mass at the amount of time elapsed
since that job entered the system,  whereas 
the server measure $\nu^N$ keeps track of jobs currently in service 
and has a unit Dirac delta mass at the age of each such job,
where the age is the amount of time elapsed since the job entered service.  
For analytical purposes, 
it turns out that the queue measure itself is more conveniently represented in terms of a 
potential queue  measure $\eta^N$, which keeps track of the times elapsed since entry into the system of
all jobs (whether or not they have  abandoned or entered service), and not only 
of   jobs  currently in the queue, as well as the total number $X^N$ of jobs in system.
 Under fairly general conditions on the service and patience distributions (see
Assumption \ref{ass-main}), 
it was shown  in \cite{KanRam10} that when the average arrival rate or traffic intensity converges to $\lambda > 0$, 
 the rescaled state descriptor $(X^{(N)}, \nu^{(N)}, \eta^{(N)})$  
converges to a deterministic limit $(X,\nu,\eta)$,  where $\nu$ and $\eta$ are characterized as the unique weak solutions 
to a nonlinearly coupled system of  deterministic measure-valued transport equations, subject to discontinuous boundary conditions, 
which we refer to as the fluid equations (see Definition \ref{def-fleqns}).

In this  work we study the long-time behavior of the  solution $(X,\nu,\eta)$ to the measure-valued fluid equations
 obtained in \cite{KanRam10} under   Assumption \ref{ass-main} and 
the assumption that the fluid equations admit a unique 
invariant state (equivalently, fixed point). We 
study the subcritical, critical and supercritical regimes, characterized, respectively,
by the regions where  $\lambda < 1$, $\lambda=1$ and $\lambda > 1$, additionally assuming 
in the subcritial and supercritical regimes  that  
the hazard rate function of the service time distribution is either decreasing or bounded
away from zero and infinity. (Here, and in the sequel, we will say decreasing to mean
non-increasing.)  
 Our main results are summarized in Theorem \ref{th-main2}. 
Specifically, we show 
  that when $\lambda \neq 1$, from any initial condition, the solution to the fluid equations converges to the
unique invariant state in the large time limit
and when $\lambda = 1$ and the hazard rate function is decreasing then the total mass of $\nu$ (which represents the
mass of busy servers in the fluid system) converges to $1$.
In all cases above, we show that  the fluid equations have a unique invariant distribution
(or equivalently, unique {\it random} fixed point, to use a term introduced
later in this paper; see Definition \ref{def-RFP}).  This 
crucially implies  that the stationary distributions of the $N$-server
dynamics converge to the invariant state of the fluid dynamics,  the proof of which
was one of the motivations  of this work. 
 In particular, as elaborated in Remark \ref{rem-fix}, it is the uniqueness of
 the invariant distribution (or random fixed point) for the fluid dynamics, rather than just the uniqueness
 of an invariant state, that is relevant for the convergence of stationary distributions of the $N$-server
 dynamics.
 In the absence of reneging, such long-time convergence results 
were established for a single-class system in Proposition 6.1 of \cite{KasRam11}
for the subcritical regime and in Theorem 3.9 of  \cite{KasRam11} for 
 the critical regime, with the  latter requiring an additional finite second moment assumption.  
 In the presence of reneging, 
although the system is in a sense  more stable (e.g., the system is also stable in the supercritical regime, making it of particular interest), 
certain monotonicity properties are lost and the fluid equation
dynamics are considerably more complicated, making the analysis significantly more challenging.   

The proof of convergence in the subcritical regime
is  obtained via a direct analysis of the fluid equations (see Section \ref{subs-pf2subcrit}). 
The proofs in the critical and supercritical cases are considerably more subtle, and  
rely on  rather different arguments under the two sets of assumptions.
When the hazard rate function of the service distribution is decreasing, we use 
a reformulation in terms of renewal equations, in conjunction with certain recursive estimates, and
the convergence of the measure-valued state processes is with respect to the weak topology
(see Section \ref{subs-pf2supcrit}). These arguments are inspired by
those used in the work \cite{ZL}, which studies the long-time behavior of fluid equations
for the GI/G/N+G model under
the assumption that the service time distribution has a concave or convex renewal
function (which is implied by decreasing hazard rate functions).
However, the fluid equations of \cite{ZL} are based on a different
measure-valued state representation, involving residual service times of customers
rather than ages, and moreover, convergence is established in \cite{ZL}
for the queue process, not the measure-valued process.
Thus the results of \cite{ZL} do not directly apply.  Moreover, 
we also need to establish additional estimates to prove convergence of 
the measure-valued  process $\nu$ in the supercritical case.

The arguments used when the hazard rate function $h^s$ of the service distribution is bounded away from zero and infinity
are of a completely different nature. These results address a class of distributions not covered
by related results in the literature for many-server systems.
They are based on the analysis of weak solutions to partial differential equations, and entail showing that
an extended relative entropy functional (that takes in arguments that are not necessarily probability measures)
 serves as a Lyapunov functional for the dynamics. 
As a result, the convergence of the measure-valued state processes is with respect to the 
stronger total variation topology.
The analysis here  is somewhat reminiscent of the study of age-structured population models arising in  biology 
 (see, e.g.,  \cite{CanCarCua13,PerBook0,MicMisPer04}). 
Indeed, although the server measure $\nu$ need not have a density, and  in fact will typically not have a continuous density,    
a purely formal derivation (see Section \ref{subsub-formal}) shows that the density of $\nu$ satisfies a 
partial differential equation  (PDE) that is similar to 
age-structured population models.  However, our fluid equation differs 
crucially from such models in several  aspects that make it in some ways harder to analyze.
One issue,  as elaborated in Remark \ref{rem-pde} is that the service time hazard rate function,
which appears as a coefficient in the PDE, is not integrable on $[0,\infty)$, and thus the relative entropy method
  developed in \cite{MicMisPer04} does not apply.
 But, more significantly,  a key complicating factor is the fact that
 the boundary condition for $\nu$ is discontinuous
and assumes a different form depending on whether
the total mass of $\nu$ is less than or equal to, $1$.  In fact, a significant challenge in the analysis of the critical
and supercritical regimes is the  to control the oscillations of the total mass of $\nu$ below $1$.

Finally,  we also analyze the long-time behavior of  fluid equations for a multiclass model under a nonpreemptive priority policy,
which was formulated in  \cite{AtaKasShi13} and used therein to establish asymptotic optimality of the  policy when the reneging distribution is exponential (see Definition \ref{def-fluidmulti}). 
In the case that
service time distribution is  class-independent and satisfies the same conditions as above, 
 reneging times are exponential, but possibly class-dependent,
and the fluid equations
have a unique invariant state, 
we establish (in Theorem \ref{th2}) uniqueness of the random fixed points and
analogous long-time convergence results in the supercritical regime. 
It should be mentioned here that   in the queueing context, other  works that have studied long-time behavior of measure-valued fluid equations using Lyapunov functionals
  include  \cite{PagTanFerAnd12,FuWil21, PuhWil16}. All of these works  focus on the dynamics of residual times for jobs in 
  bandwidth sharing and processor sharing models, which have a different structure from the measure-valued equations arising
  from our fluid equations.

An interesting open problem for future investigation would be to determine  precisely
the  full class of service distributions for which such  
long-time convergence holds, and  also whether there is a unified proof for all cases,
at least in the  supercritical regime. 
In addition, in the critical regime, a more complete study of the convergence of the state process even  under the conditions imposed here would also be
of interest. 
Moreover, the techniques developed here may be potentially used to establish such convergence results for more general
many-server systems, including load-balancing systems with general service distributions, where the fluid limits are
described in terms of a system of coupled  measure-valued equations \cite{AghRam18} or partial
differential equations \cite{AghLiRam17}, and uniqueness of the fixed point holds under general conditions \cite{AgaRam20}.
For the multiclass model, it is of interest to investigate broader conditions,
such as class-dependent service distributions and less restrictive assumptions
on the hazard rate, under which convergence holds. This would also allow to treat
asymptotic optimality of the aforementioned index rule in broader settings.

 \subsection{Ramifications for stationary distributions of $N$-server queues}

   The results of this paper  also  shed insight into the  (law-of-large-numbers) scaled limit of  stationary
distributions of $N$-server queues for a much broader class of service distributions.  
More precisely, it follows from Theorem 3.2 and Theorem 7.1 of
\cite{KanRam12}  that  the measure-valued state dynamics $(X^{(N)}, \eta^{(N)}, \nu^{(N)})$ for each $N$-server system describe an ergodic
Feller process with a unique stationary distribution, whereas  Theorem 3.3 of \cite{KanRam12} shows 
that the sequence of scaled stationary distributions $\pi^{(N)}$ of the normalized state is tight. 
Moreover, the latter  theorem also states that 
 any subsequential limit of $\pi^{(N)}$ must coincide with the 
(deterministic) invariant state of the fluid equations, whenever the latter is unique. 
However, there is a gap in the proof of this statement in \cite{KanRam12}.
{\em A priori} one only knows that any subsequential limit of the scaled stationary distributions of $N$-server
queues with reneging is a 
random fixed point of the fluid equations (see Definition \ref{def-RFP}), and  not that it is necessarily
equal to a deterministic fixed point. 
However, as shown in Proposition \ref{cor-randomfp} of the present paper,
when there is convergence of the fluid equations to a unique invariant state  from any initial condition 
 (or, when $\lambda = 1$,  just convergence of $\eta_t$ and the fraction of busy server servers), 
it follows that the set of random fixed points is in fact
equal to the Dirac delta measure at the unique invariant state, thus closing the
gap in Theorem 3.3 of \cite{KanRam12}.

Our work in the multi-class setting also closes an exactly analogous gap
present in Theorem 4.4 of \cite{AtaKasShi13}. 
Indeed, one of the  auxiliary goals of this  work
is to (partially) fix the gaps in these proofs, under the additional assumptions on the service
distribution imposed herein (see Remark \ref{rem-fix} for further elaboration of this point).
In the case of \cite{AtaKasShi13}, the gap also affects the validity of Theorem~5.1
there, regarding the asymptotic optimality of an index policy,
referred to as the $c\mu/\theta$ rule, which was introduced in \cite{AGS10}.
The results obtained in this paper validate the asymptotic optimality result,
Theorem 5.1 of \cite{AtaKasShi13} under the additional assumption that the service time
distributions do not depend on the class. (Note, however, that there is no problem with
the validity of the asymptotic optimality results of the $c\mu/\theta$ rule stated in \cite{AGS10}
and \cite{AGS11}, which deal with the case of exponential service time distributions. 
Also, note recent developments  on this policy under various additional
settings in \cite{LSZZ}).   Finally, we note that 
limits of stationary distributions of many-server systems in the (so-called Halfin-Whitt) diffusive regime have  been considered in 
\cite{gam13,AghRam19,AghRam20} in the absence of abandonment, and in 
\cite{gar02,dai14,gam12} in the presence of abandonment.

 \subsection{Organization  of the rest of the Paper}
                   
In Section \ref{subs-fleq} we introduce the fluid equations in the single-class setting,
and in Section \ref{subs-invstates} define their invariant states.
In Section \ref{sec-main} we state our assumptions and  the main results,
and provide the proofs in Section \ref{sec-proofs}. 
Finally, in Section \ref{sec-multiclass} we introduce the multiclass fluid equations
and establish our convergence results in that setting. First, in Section \ref{subs-not}, we
introduce common notation that is used throughout the paper.

\subsection{Common Notation and Terminology}
\label{subs-not}

The following notation will be used throughout the paper.
 $\Z$ is the set of integers, $\N$ is the set of strictly positive integers, $\R$ is set of real numbers,
$\R_+$ the set of non-negative real numbers.
For $a, b \in \R$, $a \vee b$  denotes  the maximum of $a$ and $b$,
$a \wedge b$  the minimum of $a$ and $b$ and the short-hand $a^+$ is used for $a \vee 0$.
Also, given a set $A$, we will use $1_A$ to denote the indicator function, which is $1$ on
  $A$ and zero otherwise.

Given any metric space $E$, $\ecb$ and $\ecc$  are, respectively,
 the space of bounded,  continuous functions and
 the space of continuous real-valued functions with compact support defined on
$E$, while  $\eocd$ is the space of real-valued,
once  continuously differentiable functions on $E$, and  $\eocdc$
is the subspace of functions in $\eocd$ that have compact support.
The subspace of functions in $\eocd$ that, together with their
first derivatives, are bounded, will be denoted by $\eocdb$. For
$H\le\infty$, let 
 ${\mathcal L}^1[0,H)$ and $\lloc[0,H)$, respectively, represent the spaces of
integrable and locally integrable functions on $[0,H)$, where a locally integrable function $f$ on $[0,H)$
is a measurable function on $[0,H)$ that satisfies
$\int_{[0,a]}f(x)dx<\infty$ for all $a<H$.
Given any c\`{a}dl\`{a}g,
real-valued function $f$ defined on $[0,\infty)$, we define $\nrm{f}_T := \sup_{s \in [0,T]} |f (s)|$ for every $T < \infty$,
and let $\nrm{f}_\infty := \sup_{s \in [0,\infty)}
|f(s)|$, which could possibly take the value $\infty$.
In addition,
the support of a function $f$ is denoted by $\supp(f)$. Given a
nondecreasing function $f$ on $[0,\infty)$, $f^{-1}$ denotes the inverse
function of $f$, defined precisely as 
\be \label{inverse} f^{-1}(y)=\inf\{x\geq 0: f(x)\geq y\}. \ee
For each differentiable function $f$ defined on $\R$, $f'$ denotes the first derivative of $f$. For each function $f(t,x)$ defined on $\R \times \R^n$,
  we will use both $f_x$ and $\partial_x f$ to  denote the partial derivatives of $f$ with respect to $x$, and likewise,
both $f_t$ and $\partial_t f$ to denote the partial derivatives of $f$ with respect to $t$.
 We use $\f1$ to denote the function that is identically
  equal to $1$.
We will mostly be interested in the case when
 $E = [0,H)$ and  $E = [0,H) \times \R_+$, for some $H \in (0,\infty]$.
To distinguish these cases, we will usually use $\fsp$ to denote generic  functions
on $[0,H)$  and $\newf$ to denote generic 
 functions on $[0,H) \times \R_+$.  By some abuse of notation,
given $\fsp$ on $[0,H)$, we will sometimes
also treat it as a function on $[0,H) \times \R_+$ that is constant in the second
  variable.

We use  ${\mathcal P}(E)$ and $\eem$ to denote, respectively, 
the space of  Radon measures on a metric space $E$, endowed with the Borel $\sigma$-algebra, and let 
 $\eemf$ denote the subspace of finite measures in $\eem$, and
   ${\mathcal M}^c_F(E)$  the subspace of continuous measures  (i.e., measures that do not charge points) in $\eemf$.
The symbol $\delta_x$ will be used to denote the measure with unit mass at the point $x$ and,
with some abuse of notation, we will use $\zerof$ to denote the identically zero
Radon measure on $E$.  When $E$ is an interval, say $[0,H)$, for notational conciseness,
  we will often write ${\mathcal M}_F[0,H)$ or
    ${\mathcal M}_F^c[0,H)$ instead of ${\mathcal M}_F([0,H))$ or
          ${\mathcal M}_F^c([0,H))$,  respectively.
           For any Borel measurable function $\fsp: [0,H) \ra \R$ that is integrable
with respect to $\xi \in  \mm$, we  often use the short-hand notation
\[ \lan \fsp, \xi \ran := \int_{[0,H)} \fsp(x) \, \xi(dx),
  \]
    and likewise, 
  for any Borel measurable function $\newf: [0,H) \times [0,\infty) \ra \R$ and $t > 0$ such that 
  $x \mapsto \newf(\cdot, t)$ is integrable 
  with respect to $\xi \in  \mm$, we  often use the short-hand notation
  \[ \lan \newf(\cdot, t), \xi \ran :=  \int_{[0,H)} \newf (\cdot,t) d\xi = \int_{[0,H)} \newf(x,t) \, \xi(dx). 
      \]
    We also let  ${\mathcal P}(E)$ denote the space of probability measures on $E$, equipped with the Borel $\sigma$-algebra. 
          
          For any  measure $\mu \in {\mathcal M}_F[0,H)$, 
we define
\be
\label{def-cdfmu}
  F^\mu (x) := \mu[0,x],  \quad x \in [0,H), 
\ee
and we define  $(F^\mu)^{-1}$ to be its  right-continuous inverse:
\be
\label{def-Finverse}
(F^\mu)^{-1} (y) = \inf \{ x > 0: F^\mu(x) \geq y \}.
\ee
Also, given $\mu, \mu_t, t \in [0,\infty),$
       in ${\mathcal M}_F[0,H)$, 
       we will use the notation $\mu_t \Rightarrow \mu$ to denote weak 
       convergence: 
       \[  \lim_{t \rightarrow \infty} \langle \fsp, \mu_t \rangle = \langle \fsp, \mu\rangle,
       \quad \forall \fsp  \in {\mathcal C}_b [0,H).
         \]
         We will also on occasion use the total variation distance on ${\mathcal M}_F[0,H)$,
            denote by $d_{\text{TV}}(\mu,\nu) := 2\sup_{A \in {\mathcal F}} |\mu(A) - \nu(A)|$, where
            ${\mathcal F}$ is the Borel $\sigma$-algebra on $[0,H)$.

              Given a Polish space $\mathcal H$, Let $ {\mathcal D}_{\mathcal H}[0,\infty)$ denote the space of $\mathcal H$-valued, c\`{a}dl\`{a}g functions on $[0,\infty)$ and $\incspace$ denote the subset of non-decreasing functions
$f \in {\mathcal D}_{\R_+}[0,\infty)$ with $f(0) = 0$. Let $\calD^+_{\R^J}(\R_+)$ denote the subset of functions in  $\calD_{\R^J}(\R_+)$ 
                    that are nonnegative and nondecreasing componentwise.

\section{Fluid Equations and Random Fixed Points}
\beginsec

\subsection{Fluid Equations}
\label{subs-fleq}

We now describe the fluid  equations  introduced in 
\cite{KanRam12}. 
Let $G^s$ and $G^r$ denote the cumulative distribution functions of the
service time and patience
time distributions, respectively.
Throughout, we make the following standing assumptions on $G^s$ and $G^r$
 and  let   $\bar{G}^s =  1 - G^s$ and
$\bar{G}^r =  1 -  G^r$  denote the corresponding
 complementary cumulative distribution functions.
 Recall that we abbreviate lower semicontinuous as lsc.

\begin{assumption}
  \label{ass-main}
 The cumulative distribution functions $G^r$ and $G^s$ satisfy
  $G^r(0+) = G^s(0+) = 0$,
  and are both absolutely continuous on $[0,\infty)$ 
  with densities $g^r$ and $g^s$
  that satisfy the following properties: 
  \begin{enumerate}
    \item 
The mean patience and service times are finite: in particular, 
\be
\label{def-mean1}
\theta^r := \int_{[0,H^r)} x g^r(x) \, dx = \int_{[0,\infty)} \bar{G}^r(x) \, dx   <\infty,
\ee 
and, we normalize  units so that 
\be
\label{def-mean2}
\int_{[0,\infty)} x g^s(x) \, dx = \int_{[0,H^s)} \bar{G}^s(x) \, dx =  1,
\ee
where
\begin{eqnarray} 
  H^s  & := &  \sup\{x \in [0,\infty): G^s(x) <1 \},\\
  H^r & := &  \sup\{x \in [0,\infty): G^r(x) <1 \},
\end{eqnarray}
denote the right-end of the supports of the measures corresponding
to $G^s$ and $G^r$, respectively.
\item
There exists $\bar{H}^s<H^s$
such that $h^s:=g^s/\bar G^s$ is either bounded or lsc
on $(\bar{H}^s,H^s)$, and likewise, there exists $\bar{H}^r<H^r$ such that $h^r:=g^r/\bar G^r$ is either bounded or lsc on $(\bar{H}^r,H^r)$.
  \end{enumerate}
\end{assumption}

\begin{remark}\label{rem20}
{\em
Strictly speaking, $g^s$ and $g^r$ (and thus $h^s$ and $h^r$)
are determined only almost everywhere.
The convention implicitly adopted in the above statement is that
$h^s$ (respectively, $h^r$) is almost everywhere (a.e.)  equal to a function from $\R_+$ to itself
that is bounded or lsc.
}
\end{remark}

At any time $t \geq 0$, the state of the fluid system is represented by  a triplet 
$(\fx(t),\fmeas_t,\freneg_t)$, where  $\fx(t)$ represents the total mass of jobs in system at time $t$,
including those in queue and those in service, $\fmeas_t$ is the fluid age measure, which is a sub-probability
measure on $[0,H^s)$   that assigns to any interval $(a,b)$ the (limiting)  fraction   
  of servers  for whom the job currently in service has
  been in service for a number of time units lying in that interval,
  and  $\freneg_t$ is the fluid potential queue measure, which is a finite measure on $[0,H^r)$ that
    to any interval $(a,b) \subset [0,\infty)$ assigns the mass (or normalized limit) of jobs that have arrived
      by time $t$ and whose patience lies in that interval  (irrespective of whether or not
       they have entered service or departed the system by time $t$).
       Note that 
        the total fraction of busy servers at time
       $t$ is  $1  - \langle  \f1, \nu_t\rangle$, which is zero if $X(t) \geq 1$ and
        $1 - X(t)$, otherwise.  This is captured succinctly by the relation
        $1 - \langle  \f1, \nu_t\rangle = [1 - X(t)]^+$.
        
      The input data for the fluid equations includes the arrival rate $\lambda$,
        and the initial conditions, consisting of the total
   initial  mass in system,
    and     the initial (fluid) age and potential queue measures. 
 Then the space of possible initial conditions for the fluid equations is given by

\be
\label{def-newspace}
\newcspace:= \left\{\begin{array}{cc} & 
(\tilde{x}, \tilde{\nu}, \tilde{\eta}) \in  \R_+ \times \mmfs \times \mmfr: \\ &
1 - \lan \f1, \tilde{\nu} \ran = [1-\tilde{x}]^+ \end{array} \right\}.
\ee

We now give a precise formulation of the fluid equations introduced in \cite{KanRam12}  with $E(t)=\fe^\lambda (t) := \lambda t$ for $t \geq 0$ therein. 
These equations will also involve the  queue process 
 $Q(t)$, which  represents  the  total mass in queue (awaiting service) at time $t$,  and the non-decreasing processes 
$\fd(t), \fk(t)$, $\fs(t)$ and $\fr(t)$
represent,  respectively, the cumulative mass of departures from the queue,  entry into service, and respectively, the 
potentially reneged  and actually reneged jobs from the queue in the interval $[0,t]$.

\begin{definition} {\bf (Fluid Equations)}
\label{def-fleqns} {\em Given hazard rate functions $h^r$ and $h^s$, the c\`{a}dl\`{a}g function $(\fx,
\fmeas,\freneg)$ defined on $[0,\infty)$ and taking values in $\R_+ \times
\mmfs \times \mmfr$ is said to solve the  fluid equations  with arrival rate $\lambda\geq 0$ and initial condition 
$(\fx (0), \fmeas_0,\freneg_0) \in \newcspace$    if for every $t \in [0,\infty)$, we have} 
\be
\label{cond-radon}
S(t) :=  \int_0^t  \lan h^r, \freneg_s \ran \, ds  < \infty,\qquad  \fd(t) :=  \int_0^t  \lan h^s, \fmeas_s \ran \, ds  < \infty,
\ee
{\em and  the following relations are satisfied: for every $\newf \in \newocdcpms$, }
\begin{eqnarray}
\label{eq-ftmeas}
\ds \lan \newft, \fmeas_t \ran  & = & \ds \lan  \newf(\cdot,0), \fmeas_0 \ran +
\int_0^t \lan \dtnewfs +  \dxnewfs,  \fmeas_s \ran \, ds  \\
 & & \nonumber
\quad  - \ds \int_0^t \lan  h^s(\cdot) \newfs,  \fmeas_s \ran \, ds+ \int_0^t \newf (0,s) \, d\fk (s),
\end{eqnarray}
{\em where }  
\be
\label{eq-fk}
\fk(t) = \lan \f1, \fmeas_t \ran - \lan \f1, \fmeas_0 \ran + D(t); 
\ee
{\em for every} $\newf \in \newocdcpmr$, 
\begin{eqnarray}
 \label{eq-ftreneg}
\ds \lan \newf(\cdot,t), \freneg_t \ran  & = & \ds \lan  \newf(\cdot,0), \freneg_0 \ran +
\int_0^t \lan \newf_s(\cdot,s)  + \newf_x(\cdot,s), \freneg_s \ran \, ds \\ 
 & & \nonumber
\quad  - \ds \int_0^t \lan  h^r(\cdot) \newf(\cdot,s),  \freneg_s \ran \, ds+ \lambda \int_0^t \newf (0,s) \,
  ds;
\end{eqnarray}
{\em with the non-idling constraint }
\be
\label{eq-fnonidling}
 1 - \lan \f1, \fmeas_t \ran = [1 - \fx(t)]^+, 
 \ee
{\em where }
\begin{eqnarray}
\label{eq-fx}
\fx (t) &  = & \fx (0) + \lambda t -
D(t) - \fr(t), 
\end{eqnarray} 
 with 
 \begin{eqnarray}
   \label{fr}
   \fr(t) &=& \int_0^t \left(\int_0^{\fq(s)}h^r((\fnsf)^{-1}(y))dy\right) ds, 
 \end{eqnarray}
 {\em where  recall $\fntf(x) :=\freneg_t[0,x],$  and $(\fntf)^{-1}$ denotes the right-continuous
    inverse defined in \eqref{def-Finverse}, and } 
\begin{eqnarray}  
\label{fq}
\fq(t)& =& \fx(t) - \lan \f1, \fmeas_t \ran,
\end{eqnarray}
 with $\fq$ also satisfying the inequality constraint 
\begin{eqnarray}  
\label{fqfreneg} \fq(t)&\leq &\lan \f1, \freneg_t \ran. 
\end{eqnarray}
 \end{definition}

  \begin{remark}
    \label{rem-eta}
{\em Note that if $(X,\fmeas,\freneg)$ solves the fluid equations with  arrival rate $\lambda$, and initial condition 
		$(\fx (0), \fmeas_0,\freneg_0) \in \newcspace$,  then
we also have 
$(\fx(t), \fmeas_t, \freneg_t)  \in \newcspace$ for
every $t > 0$. 
It is also true that if $ \freneg_0 \in  \mmfcr$, then we also have
$ \freneg_t \in  \mmfcr$ for every $t > 0$
(this follows from the expression for $\eta_t$ in \eqref{f4} below, from which
it is clear that if $\eta_0$ does not charge points then neither does $\eta_t$).
}
   \end{remark}

Also, note from (\ref{fq}) and (\ref{eq-fnonidling}) that for each $t\in [0,\infty)$, \be \label{fqfx}\fq(t)=[\fx(t)-1]^+. \ee For future use, we also observe that (\ref{eq-fk}), (\ref{fq}) and (\ref{eq-fx}), when combined, show that for every $t\in [0,\infty)$,
    \be \label{qt-conserve}  \fq(0)+\lambda t=\fq(t)+\fk(t)+\fr(t), \ee
    which is simply a  mass conservation equation. 
In addition, we will find it convenient to define
\be
\label{def-q}
\fb(t) := \langle \f1, \nu_t \rangle,  \quad t \geq 0, 
\ee
which represents the limiting fraction of busy servers.

\begin{remark}
  \label{rem-aux}
{\em Given a solution $(X,\nu,\eta)$, we will refer to $(D, K, R, S, Q, B)$ as auxiliary processes.}	
	
\end{remark}

We now  provide an informal, intuitive explanation for the form of the fluid
equations. Note that $\fmeas_s(dx)$ represents the amount of mass (or 
fraction of servers) that are processing jobs whose ages lie in the range $[x,x+dx)$ at time $s$, and
  $h^s(x)$ represents the conditional mean rate at which the mass of jobs with age in $[x,x+dx)$ 
  completes service  at time $s$.    Hence,  in \eqref{cond-radon}, $\lan h^s,\fmeas_s \ran$  represents the 
    departure rate of mass from the fluid system due to services at time $s$, and
  its integral, $D(t)$, is the  cumulative departure rate due to service completion in the  interval $[0,t]$. 
  By an exactly analogous reasoning, the other quantity
  $S(t) =  \int_0^t \lan h^r,\freneg_s \ran ds$ in \eqref{cond-radon} represents the cumulative
  potential reneging from the system in the interval $[0,t]$.
  However, the actual reneging rate is restricted to abandonments of those in queue.  
  Since entries into the queue take place in the order of  arrival, the age of the oldest 
  (equivalently, head-of-the line) mass in the fluid queue is $\bar{a}_s := (F^{\eta_s})^{-1}(Q(s))$, so  that
  $\eta_s[0,\bar{a}_s] = Q(s)$.  Here, recall  $F^{\eta_s}$ represents the cumulative
  distribution function of $\eta_s$. Thus, the actual reneging rate 
  at any time $s$ only counts the  mass reneging from the potential queue measure $\eta_s$ whose age lies in the
  restricted interval $[0,\bar{a}_s]$,
  rather than the entire interval $[0,\infty)$. A standard change of variable then yields the
    expression in  \eqref{fr}. 
       Next, recalling the interpretations of the  quantities  $K$, $R$ and $Q$
  stated prior to Definition \ref{def-fleqns},  
  note that equations (\ref{eq-fk}), (\ref{fq}) and  (\ref{eq-fx}) are simply
  mass conservation equations, and \eqref{eq-fnonidling} represents a  
  non-idling condition that ensures that no server can idle when there is work in the queue. 
  Moreover, the inequality (\ref{fqfreneg})
expresses the  constraint that  at any time $t$, the  fluid queue is bounded by
the total mass of the fluid potential queue measure, since the
latter  also includes mass that may have already gone into service (and possibly also departed the system) by that time, provided its 
patience  time exceeds the total time elapsed since arrival.     
  Finally, equations (\ref{eq-ftmeas})
and (\ref{eq-ftreneg})  govern the evolution of the fluid age measure
$\fmeas$ and potential queue measure $\freneg$, respectively.   In particular,  
 the second term on the right-hand-side of (\ref{eq-ftmeas}) represents the change in
 $\langle \varphi, \fmeas\rangle$ over the interval $[0,t]$ due to transport
 or shift of the ages at unit rate to the right,
 the third term accounts for changes due to departure of mass from the system due to service, and the last term
 captures changes due to new  entry into  system, which are driven by the function $K$,
   the cumulative entry into service.
 The equation  \eqref{eq-ftreneg} is exactly analogous,
 but  with $h^r$  and  the  cumulative arrivals $E^\lambda$ into the  system in place of $h^s$
 and $K$, respectively, and the third term on the right-hand side now representing departure of
 mass from the system due to potential reneging.

 We now state a result that was proved in   \cite{KanRam10,KasRam11}.
 Recall the definition of the
 % initial condition
 space  $\newcspace$ given in 
 \eqref{def-newspace}. 
   
 \begin{theorem}
   \label{th-fluid}
   Suppose Assumption \ref{ass-main} holds and fix $\lambda \geq 0$, and
    $(\fx (0), \fmeas_0,\freneg_0) \in \newcspace$.
    Then there is at most one solution 
     to the fluid equations with  arrival rate $\lambda$ and  initial condition
     $(\fx (0), \fmeas_0,\freneg_0)$, and if $\freneg_0 \in  \mmfcr$, then
     there also  exists
     a  continuous solution   $(X,\fmeas,\freneg) = \{(\fx(t), \fmeas_t, \freneg_t), t \geq 0\}$ with  arrival rate $\lambda$ and initial condition     $(\fx (0), \fmeas_0,\freneg_0)$. 
     Moreover, given any solution  $(X,\fmeas,\freneg) = \{(\fx(t), \fmeas_t, \freneg_t), t \geq 0\}$ to the
     fluid equations associated with  $\lambda$ and  
     $(\fx (0), \fmeas_0,\freneg_0) \in \newcspace$, the following properties hold: \\
(i)   for any bounded or nonnegative measurable function $\fsp$ on $[0,\infty)$ and for  $\fsp = h^r$, 
    for every $t \geq 0$, 
\begin{equation}
	\int_{[0,H^r)} \fsp (x) \eta_t (d x) = \int_{[0,H^r)} \fsp (x+t)\frac{\bar{G^r}(x+t)}{\bar{G^r}(x)}\eta_0(dx)+
	\int_0^t \fsp (s) \bar{G^r}(s) \la ds; \label{f4}
\end{equation}
(ii) for  any bounded or nonnegative measurable function $\fsp$ on $[0,\infty)$ and for $\fsp = h^s$, 
    for every $t \geq 0$, 
\begin{equation}
	\int_{[0,H^s)} \fsp (x) \nu_t (d x) = \int_{[0,H^s)} \fsp (x+t)\frac{\bar{G^s}(x+t)}{\bar{G^s}(x)}\nu_0(dx)+
	\int_0^t \fsp (t-s) \bar{G^s}(t-s) d K(s), \label{f5}
\end{equation}
       with $\fk$ equal to the auxiliary process defined in \eqref{eq-fk}
       of the fluid equations; \\
       (iii) if $Q$ and $B$ are the associated auxiliary processes 
         defined in \eqref{fq} and \eqref{def-q} respectively, 
   then $\fk$ is an absolutely continuous function and for a.e.\ $t \geq 0$,
   the derivative $\fk^\prime$ of $\fk$ satisfies 
   \begin{equation}
     \label{fkprime}
     \fk^\prime (t) = k(t):= \left\{
     \begin{array}{rl}
       \lambda & \mbox{ if }  B(t)  < 1,    \\
       \lambda  = \langle h^s, \nu_t \rangle  & \mbox{ if  } B(t) = 1 \mbox{ and } \fq(t) = 0, \\
       \langle h^s, \nu_t \rangle  &  \mbox{ if  }  B(t) = 1   \mbox{ and } \fq(t) > 0.  
     \end{array}
     \right.
   \end{equation}
  \end{theorem}
 \begin{proof} 
   Uniqueness of the solution to the fluid equations follows
   from Theorem 3.5 of \cite{KanRam10}   since  $(\fx (0), \fmeas_0,\freneg_0) \in \newcspace$
     implies  $(\fe^\lambda, \fx(0),\fmeas_0, \freneg_0)$ lies in the space $\newspace$ therein, where recall $\fe^\lambda(t) = \lambda t$. 
   Likewise, existence  of a solution with  arrival rate $\lambda$ and  initial condition  $( \fx(0),\fmeas_0, \freneg_0) \in \newcspace$ with $ \freneg_0 \in  \mmfcr$
   can be deduced from Theorem 3.6 of \cite{KanRam10}, once we justify that 
  the conditions of that theorem 
   are satisfied in the present setting. 
   First, it is
   not hard to see that for any arrival rate $\lambda\geq 0$ and initial condition $( \fx(0),\fmeas_0, \freneg_0) \in \newcspace$ with $ \freneg_0 \in  \mmfcr$   one can construct
   a sequence of $N$-server systems with Poisson $(N\lambda)$ arrival process $E^N$ and initial condition $(\fx^N(0), \fmeasn_0, \frenegn_0)$ such that 
   Assumption 3.1  of \cite{KanRam10}  is satisfied. 
   Second,  note that since $\freneg_0$ is
   a continuous measure, and   $E^\lambda$ is continuous,  
    Assumption 3.2 of \cite{KanRam10} is also satisfied.
    Finally, Assumption 3.3 of \cite{KanRam10} is a direct consequence of
    Assumption \ref{ass-main} of this paper, and thus 
  the application  of Theorem 3.6 of \cite{KanRam10} is justified.

We now turn to estabishing the properties  of any solution $(X,\fmeas,\freneg)$ to the fluid equations. 
   First,  note that the forms of both \eqref{eq-ftreneg} and \eqref{eq-ftmeas} are analogous to that of 
   (4.2) in \cite{KasRam11}, and therefore the integrability conditions in \eqref{cond-radon} imply 
   that   (4.1) of \cite{KasRam11} holds.  
   Thus, \eqref{f4} and  \eqref{f5} for $\fsp  \in  {\mathcal C}_c[0,H^r)$ and
     $\fsp \in {\mathcal C}_c[0,H^s)$, respectively, 
   follow from Theorem 4.1 of \cite{KasRam11}.   By using a standard approximation argument,
     namely  representing indicators of finite open intervals in $\R_+$ as
       monotone limits of continuous
       functions with compact support and appealing to the monotone class theorem, it follows that both equations
       in fact hold for any bounded measurable or nonnegative measurable $\fsp$.
       In particular, these equations also hold with $\fsp = h^r$ in \eqref{f4} and  $\fsp = h^s$ in \eqref{f5}.
       The latter fact is used several times in this paper. 
   
We now turn to the  proof of property (iii),  which  is similar in spirit to  (3.12) of
\cite{KasRam11}  and   Corollary 3.7 of \cite{KanRam10}. We supply the details for completeness.  
First,  note that $D,  S$ and $R$ are absolutely continuous by definition. Thus,
\eqref{eq-fx} shows that $X$ is absolutely continuous, which then implies that 
$B(t)  = \langle \f1, \nu_t\rangle   = \min (X(t), 1)$ is also absolutely continuous. In turn, by \eqref{eq-fk} and \eqref{fq},
this implies that 
$K$ is absolutely continuous. Further, \eqref{eq-fk}, (\ref{qt-conserve}), \eqref{cond-radon} and \eqref{def-q} show that for a.e.\ $t > 0$,  
\begin{equation}
  \label{kprime-temp}
  K'(t)
  = \lambda  - Q'(t) - \int_0^{Q(t)} h^r((F^{\eta_t})^{-1}(y) ) dy, \quad \mbox{ and }  \quad
K'(t) = B^\prime (t) + \langle h, \nu_t \rangle. 
\end{equation} 
We now recall the following standard fact that given  an absolutely continuous function  and a set $A$ on which it is constant,
the derivative of the function is zero for almost every $t \in A$.
Thus, for almost every $t$ in the set where 
$B$ is constant, we have 
$B^\prime(t) = 0$, and \eqref{kprime-temp} implies $K'(t) = \langle h, \nu_t \rangle$.
On the other hand, for almost every $t$  when $Q(t) = 0$, it follows that  $Q^\prime(t) = 0$ and hence,
by \eqref{kprime-temp} that $K'(t)  = \lambda$.  Thus, when both $B(t) = 1$ and $Q(t) = 0$, 
$K^\prime(t) = \lambda = \langle h, \nu_t \rangle$. 
The remaining claims in \eqref{fkprime} then follow  from the observations that when 
$B(t) < 1$, one has $Q(t) = 0$  and when $Q(t) > 0$, one has  $B(t)$  equal to the constant one,
both of which are easily deduced from  \eqref{eq-fnonidling} and \eqref{fq}. 
\end{proof}

 We now state  a simple result on the action of time-shifts on solutions to the fluid equations. 
  To state the result, which 
  was formulated as  Lemma 3.4 of \cite{KanRam12}, 
  we will need the following notation: for any $t\in [0,\infty)$, define
    \[
    \begin{array}{rclrclrclr}
    \fk^{[t]} &:=& \fk(t+\cdot)-\fk(t), \qquad \qquad \fx^{[t]} &:=& \fx(t+\cdot), \qquad \qquad \fmeas^{[t]} &:= &\fmeas_{t+\cdot}, \\
    \fr^{[t]} &:=& \fr(t+\cdot)-\fr(t),  \qquad \qquad \quad \freneg^{[t]} &:= &\freneg_{t+\cdot}, \qquad \qquad \quad \fq^{[t]} &:=& \fq(t+\cdot). 
    \end{array}
    \]

\begin{lemma}[Lemma 3.4 of \cite{KanRam12}] \label{lem:shift}
  Suppose Assumption \ref{ass-main} holds. 
  Suppose  $(\fx,\fmeas,\freneg) = \{(\fx(s), \fmeas_s, \freneg_s), s \geq 0\} \in  {\mathcal D}_{\newcspace}[0,\infty)$ 
    solves the  fluid equations with arrival rate $\lambda$ and initial condition 
    $(\fx (0), \fmeas_0,\freneg_0) \in \newcspace$, then for any $t > 0$,
    $(\fx^{[t]}, \fmeas^{[t]},\freneg^{[t]})$ solves the fluid equations with arrival rate $\lambda$ and initial condition 
    $(\fx (t), \fmeas_t,\freneg_t) \in \newcspace$,  but with $\fk, \fr$ and $\fq$ replaced with
    $\fk^{[t]}, \fr^{[t]}$ and $\fq^{[t]}$, respectively. 
\end{lemma}

As in \cite{KanRam12}, we leave the proof to the reader, since it 
can be verified by just rewriting the fluid equations and invoking the
uniqueness result stated in Theorem \ref{th-fluid}.

 \subsection{Invariant States and Random Fixed Points of the Fluid Equations}
 \label{subs-invstates}

Let $\fmeass$ and $\frenegs$ be  Borel probability measures on $[0,\infty)$ defined as
follows:
\begin{eqnarray}
\label{def-invmeas}
\fmeass [0,x) & := & \int_0^x \bar{G}^s(y) \, dy, \qquad x \in [0,H^s), \\
\label{def-invrenegs}
\frenegs [0,x) & := & \int_0^x \bar{G}^r(y) \, dy, \qquad x \in [0,H^r).
\end{eqnarray}
Note that $\fmeass$ and $\frenegs$ are well defined due to Assumption \ref{ass-main}.
For $\lambda \geq 1$, define the set $\fplambda_\lambda$ as follows:
\be
\label{eq-fp}
 \fplambda_\lambda := \left\{x \in [1,\infty): G^r \left( \left(F^{\lambda \frenegs}
    \right)^{-1} \left((x-1)^+\right) \right) = \frac{\lambda - 1}{\lambda} \right\}, 
\ee
and let \[ x_l^\lambda:= \inf\left\{x\in [1,\infty):\ x\in \fplambda_\lambda\right\} \qquad \mbox{and} \qquad
  x_r^\lambda:= \sup\left\{x\in [1,\infty):\ x\in  \fplambda_\lambda \right\}.\]
By \eqref{def-invrenegs}, the map $x\to\eta_*[0,x)$ is strictly increasing on $[0,H^r)$, and therefore
$(F^{\lambda \frenegs})^{-1}$ is continuous. Since $G^r$ is also continuous,
we have $\fplambda_\lambda = [x_l^\lambda,x_r^\lambda]$ is non-empty. 
Let ${\mathcal I}_{\lambda}$ be the invariant manifold for the fluid equations, defined by
\begin{equation}
\label{eq-invman}
{\mathcal I}_{\lambda} :=
\left\{
\begin{array}{ll}
\left\{ (\lambda, \lambda \fmeass, \lambda \frenegs) \right\} &
\mbox{ if } \lambda < 1, \\
\left\{ (x_*, \fmeass, \lambda \frenegs): x_* \in  \fplambda_\lambda
\right\}&
\mbox{ if } \lambda \geq 1.
 \end{array}
 \right.
\end{equation}
 Our study of  the critical and super-critical regimes will be carried out under the 
  following additional assumption on the  invariant manifold ${\mathcal I}_{\lambda}$.

\begin{assumption}\label{ass-unique}
The set ${\mathcal I}_{\lambda}$ has a single element,  which we express as $z_*^\lambda = (x_*^\lambda, (\lambda \wedge 1)\fmeass,\lambda \frenegs)$, where $x_*^\lambda =\lambda$ when $\lambda < 1$ and $x_*^\lambda$ is the unique element of $\fplambda_\lambda$ when $\lambda \geq 1$.
\end{assumption}

\noi Note that Assumption \ref{ass-unique} imposes a non-trivial restriction only when $\lambda \geq 1$. 
As stated in Lemma 3.1 of \cite{KanRam12}, a sufficient condition for
Assumption \ref{ass-unique} to hold when $\lambda > 1$ is for the equation
$G^r(x) = (\lambda - 1)/\lambda$ to have a unique solution. 
 
 Whereas Assumption \ref{ass-unique} guarantees
  a unique deterministic fixed point for the fluid equations, to understand
  the large-time limits of the fluid equations, it turns out to be important to also understand the collection of 
  random fixed points,  defined below. 
We first introduce the notion of a solution to the fluid
equations when the input data is random. 

\begin{definition} 
  \label{def-fleqnsri}   Given any $\newcspace$-valued 
    random element 
    $(\fx (0), \fmeas_0,\freneg_0)$ defined on some probability space
    $(\Omega, {\mathcal F}, \PP)$,  we say the c\`{a}dl\`{a}g $\newcspace$-valued 
stochastic process   $(\fx, \fmeas, \freneg) = \{(\fx(t),\fmeas_t,\freneg_t), t \geq  0\}$ is a solution to the fluid equations with arrival rate $\lambda$ and random initial condition
$(\fx (0), \fmeas_0,\freneg_0)$  if for each $\omega \in \Omega$, the  function 
$(X(\omega),\fmeas(\omega),\freneg(\omega)) = \{(X(t,\omega), \fmeas_t(\omega), \freneg_t(\omega)), t \geq 0\}$
solves the fluid equations with arrival rate $\lambda$ and initial condition  $( \fx (0,\omega), \fmeas_0(\omega),\freneg_0(\omega))$.   
\end{definition}

\begin{definition} 
  \label{def-RFP}  
       For $\lambda > 0$, a probability measure  $\mu$ on $\newcspace$ 
              is said to be a random fixed point of the fluid equations  with arrival rate $\lambda$
              if  given any  $\newcspace$-valued random element 
              $(\tilde X,\tilde \nu, \tilde \eta)$  whose law is $\mu$,
            there exists a solution 
          $(X, \fmeas, \freneg)$  to the fluid equations with arrival rate $\lambda$ and random initial condition 
          $(\tilde X,\tilde \nu, \tilde \eta)$  such that  
          for each $t\geq 0$, the law of  $(\fx(t),\fmeas_t, \freneg_t)$ is equal to  $\mu$.           
\end{definition}

\begin{remark}
  \label{rem-fp}
  {\em  Under our assumptions, a random fixed point always exists. 
    Indeed, it follows from Theorem 5.5  of \cite{KanRam12} that the set ${\mathcal I}_{\lambda}$ in \eqref{eq-invman} 
    describes the so-called
    invariant manifold (or collection of deterministic fixed points) of the fluid equations.
    Since for $\lambda \geq 0$, $\fplambda_\lambda$ is always non-empty, 
  an immediate consequence is that for any $z \in {\mathcal I}_{\lambda}$, the measure $\delta_z$ is a
  random fixed point of the fluid equations with arrival rate $\lambda$.  Moreover, 
  under Assumption \ref{ass-unique}, $\delta_{z^\lambda_*}$ is the only  random fixed point
  that is degenerate (i.e., which concentrates all its mass on one point).  A key question
  we  address in this article is to determine  conditions  under which this is in fact the only  random fixed point with
  arrival rate $\lambda$. 
  As shown in Proposition \ref{cor-randomfp} below,   a sufficient
  condition for this to hold is that any solution $(X,\nu,\eta)$ to the fluid equations with arrival rate $\lambda$ and initial condition $(X(0), \nu_0, \eta_0) \in \newcspace$ and auxiliary process $B = \langle \f1, \nu\rangle$ satisfies $\eta_t \Rightarrow \lambda \frenegs$ and
   $B_t \rightarrow  \lambda \wedge 1$,  as
  $t \rightarrow \infty$. 
 }
\end{remark}

\section{Assumptions and Main Results}
\label{sec-main}
\beginsec

We now state our main results, which require the following additional
condition on the
service distribution. 

\begin{assumption}
  \label{ass:a}
  The cumulative distribution function $G^s$ of the service distribution has a density $g^s$ and the
  hazard rate function $h^s = g^s/\bar{G}^s$ satisfies  one of the following: 
  \begin{enumerate}
  \item
    \label{ass:a1}
 The quantities $\eps_h :=\essinf_{x\geq 0}h^s(x)>0$ and $c_h :=\esssup_{x\geq 0}h^s(x)<\infty$. 
\item
  \label{ass:a2}
  The function $h^s$ is decreasing.
  \end{enumerate}
\end{assumption}

The second part of the above assumption
should be understood in the sense of Remark \ref{rem20},
namely $h^s$ is a.e.\ equal to a decreasing function from $[0,H^s)$ to $\R_+$.
Note that under both parts of the assumption,
the hazard rate function $h^s$ has a finite essential supremum.
Since the hazard rate function of any distribution is only locally integrable and never integrable on its support, both Assumptions \ref{ass:a}(\ref{ass:a1}) and \ref{ass:a}(\ref{ass:a2}) imply $H^s = \infty$.

\begin{theorem}
  \label{th-main2}
  Suppose Assumption \ref{ass-main} holds, and $\nu_*$ and  $\eta_*$ are as defined in \eqref{def-invmeas} and
  \eqref{def-invrenegs}.    Also, suppose 
  $(\fx,\fmeas, \freneg)$ solves  the fluid equations  with arrival rate $\lambda$ and initial condition $(X(0), \nu_0, \eta_0) \in \newcspace$, with auxiliary processes $(D, K, R, S, Q, B)$
   as in Remark \ref{rem-aux}.  
  Then the following is true:
  \begin{enumerate}
  \item 
    When $\lambda < 1$,
   it follows that   $(\fx_t, \fmeas_t, \freneg_t) \rightarrow (\lambda, \lambda \fmeass, \lambda \frenegs)$
  as $t \rightarrow \infty$. 
    In particular,  $\delta_{z_*^\lambda}$, with  
   $ z_*^\lambda  = (\lambda, \lambda \fmeass, \lambda \frenegs)$, is the unique random fixed point
    of the fluid equations with  arrival rate $\lambda$.
  \item   When $\lambda > 1$, 
      and Assumption \ref{ass:a}  is also satisfied, then
    $\freneg_t \Rightarrow  \lambda\frenegs$ as $t \to \iy$ and 
    \begin{enumerate} 
    \item  there exists $T < \infty$ such that 
  \begin{equation}
    \label{conv-1}
    B(t) = \langle \f1, \fmeas_t\rangle = 1, \quad \mbox{ for all } t \geq T, 
  \end{equation}
   and 
   \begin{equation}
     \label{conv-2}
     \fmeas_t \Rightarrow \fmeass \mbox{ and } \lan h^s,\nu_t\ran\to1  \quad \mbox{ as }  t \rightarrow \infty, 
   \end{equation}
   with the convergence in \eqref{conv-2} also holding in total variation when  Assumption \ref{ass:a}(\ref{ass:a1}) holds;
   \item
     if, in addition, Assumption \ref{ass-unique}  is also satisfied 
      (with $z_*^\lambda$ as defined therein), then $\delta_{z_*^\lambda}$ is the unique random fixed point of the fluid equations  
    with  arrival rate $\lambda$. 
   \end{enumerate}  
  \item  If $\lambda = 1$   
     and     Assumption \ref{ass:a}(\ref{ass:a2}) is satisfied,
     then $\freneg_t \Rightarrow  \lambda\frenegs$ and $B(t) \to 1$ as $t \to \infty$.
     If,  in addition, Assumption \ref{ass-unique} holds (with $z_*^1$ as defined therein), then 
    $\delta_{z_*^1}$ is the unique random fixed point of the fluid equations  
    with  arrival rate $1$.  
  \end{enumerate} 
\end{theorem}
\begin{proof}
 The proof of statement (1) is given in Section \ref{subs-pf2subcrit}, and  for all $\lambda \geq 0$,
    the (weak) convergence of $\eta_t$ to
    $\lambda \frenegs$ under Assumption \ref{ass:a} 
    follows from Lemma \ref{lem-reneg}.
    Now, when $\lambda >1$,  \eqref{conv-1} and \eqref{conv-2} follow from Proposition \ref{prop-superpf} and Remark \ref{rem-theta}
    when Assumption \ref{ass:a}(\ref{ass:a1}) holds, and from   Proposition \ref{prop-key} and  Lemma \ref{cor-supercritlim} when
    Assumption \ref{ass:a}(\ref{ass:a2}) holds.  Further, when $\lambda = 1$ the convergence $B(t) \rightarrow 1$ under
    Assumption \ref{ass:a}(\ref{ass:a2}) also follows from Proposition \ref{prop-key}. 
    Lastly,  the uniqueness results for the random fixed point stated in (2b) and (3) 
follow from the convergence results in (2a) and   (3) and Proposition \ref{cor-randomfp}. 
  
\end{proof}

  \begin{remark}
   \label{rem-fix} 
         {\em           
           The main application of Theorem \ref{th-main2} is to characterize the limit of the scaled
           stationary distributions of the sequence of $N$-server measure-valued state processes, and thereby 
           (partially)  fix a technical flaw  in the  convergence result stated in Theorem 3.3 of \cite{KanRam12}.

           To explain this in greater  detail, 
           let $\bar{Z}^{(N)}_* := (\bar{X}^{(N)}_*, \bar{\nu}^{(N)}_*, \bar{\eta}^{(N)}_*)$
           have the law of  the  stationary distribution of
           the measure-valued $N$-server state dynamics of an $N$-server queue with reneging introduced in \cite{KanRam10},
          when the scaled arrival process is given by  
           $\bar{E}^{(N)}$ (e.g., Poisson with a scaled arrival rate $\lambda^{(N)} > 0$).
           Existence of such a  stationary distribution was established in Theorem 7.1 of
           \cite{KanRam12}. 
           Also, let $\bar{Z}^{(N)} := (\bar{X}^{(N)}, \bar{\nu}^{(N)}, \bar{\eta}^{(N)})$ represent
           the dynamics of the fluid-scaled measure-valued
           state representation of the $N$-server queue with  initial condition 
           $\bar{Z}^{(N)}(0) = \bar{Z}^{(N)}_*$.  
           Then,  under the  assumption that
           $\bar{E}^{(N)}$ converges weakly to $E^\lambda$ for some $\lambda > 0$,  
            tightness of the sequence
   $\{\bar{Z}^{(N)}_* := (\bar{X}^{(N)}_*, \bar{\nu}^{(N)}_*, \bar{\eta}^{(N)}_*)\}_{N \in \N},$ 
            was  established in Theorem 6.2 of \cite{KanRam12}.
            Let  $\bar{Z} = (\bar{X}, \bar{\nu}, \bar{\eta})$ denote any subsequential limit.
            We now claim that then (the law of) $\bar{Z}$ must be a random fixed point of the fluid equations with
            arrival rate $\lambda$. 
            To see why the claim is true,  we invoke 
            the fluid limit theorem established in Theorem 3.6 of \cite{KanRam10}, to conclude that for any
            $t > 0$, the $N$-server fluid-scaled state process $\bar{Z}^{(N)}(t)$ (initialized at  the stationary
            distribution $\bar{Z}^{(N)}_*$)     
           converges weakly to  $Z(t)$, where
           $Z := (X,\nu,\eta)$ solves the fluid equations  with arrival rate $\lambda$ and initial condition $\bar{Z}$. 
           However, for  any $t > 0$,
           since by stationarity $\bar{Z}^{(N)}(t)$  has the same law as $\bar{Z}^{(N)}_*$,           
           it follows that the laws of their 
           corresponding weak limits, $Z(t)$ and $\bar{Z}$, must also coincide.  
           By Definition \ref{def-RFP}, this proves the claim that $\bar{Z}$ is a random fixed point.

           In the proof of
           Theorem 3.3 (of Section 6.2) in \cite{KanRam12},  it was 
           assumed without justification that $\bar{Z}$ is deterministic, and that was used to conclude 
           that  $\bar{Z}$ must belong to the invariant manifold ${\mathcal I}_\lambda$ (see Remark \ref{rem-fp}).   
           When combined with Assumption \ref{ass-unique}, this leads to the conclusion
           that  $\bar{Z} = z_*^\lambda$, thus showing that all subsequential limits coincide,
           and hence, that $z_*^\lambda$ is the weak limit of the  original stationary sequence $
           (\bar{Z}^{(N)}_*)_{N \in \N}$. 
           However, one cannot assume {\em a priori} that
           $\tilde{Z}$ is deterministic, and, as argued above, one only knows that any subsequential limit
           is a {\em random fixed point.}  To make this argument complete, which was  one of the main motivations of this paper, 
           one needs to show that there is precisely one 
           random fixed point, namely the one  concentrated at $z_*^\lambda$. 
            Theorem \ref{th-main2} does precisely this 
            for the class of service distributions satisfying Assumption \ref{ass:a},
            thus closing the gap in the proof of the convergence result in
             \cite{KanRam12} (for  service distributions in that class). 
          However, this still leaves the open 
             question of whether 
            this result remains true for a larger class of service distributions, in particular
            the entire class considered in \cite{KanRam12}.  
         }
 \end{remark}

 \begin{remark}
   \label{rem-interchange}
   {\em 
            Further, a related ancillary goal of this work is to determine whether 
            the diagram in Figure \ref{diag-comm} below commutes under general convergence
            conditions on the initial states (essentially Assumption 3.1 of \cite{KanRam10}).
           Referring to the same notation as used in Remark \ref{rem-fix},  
            the  top horizontal arrow in Figure \ref{diag-comm}
            holds due to 
            ergodicity of the $N$-server state dynamics, which  was established in
            Theorem 7.1 of \cite{KanRam12} under some additional conditions on the service
            and reneging distributions (see Assumption 7.1 therein).  
          On  the other hand, as already mentioned in Remark \ref{rem-fix}, 
           the left vertical arrow  follows from the fluid limit theorem
           Theorem 3.6 of \cite{KanRam10} (under suitable convergence assumptions on the initial data). }
           
            \begin{figure}[h]
 	\centering
 	\[ \begin{array}{lcr} 
 	  \bar{Z}^{(N)} (t) =	(\bar X^{(N)}(t),\bar \nu^{(N)}_t,\bar\eta^{(N)}_t) \qquad & 
         \stackrel{\mbox{\tiny{Thm 7.1 of \cite{KanRam12}}}}{\Longrightarrow}   & 
            \qquad 
            \bar{Z}^{(N)}_* =  (\bar X^{(N)}_*,\bar \nu^{(N)}_*,\bar \eta^{(N)}_*) \\
            \hspace{0.1in} \\
            \qquad    \qquad \,   \big\lVert & &   \big\rVert \, \qquad    \qquad \\
          \qquad   \mbox{\tiny{Thm 3.6 of \cite{KanRam10}}} & &  \mbox{\tiny{Thm 6.2  of \cite{KanRam12} and {\bf Thm 3.2}}}\\
            \qquad    \qquad	\big\Downarrow & & \big\Downarrow \qquad \qquad \\
             \hspace{0.1in} \\
 	     Z(t) = 	(X(t), \nu_t, \eta_t) \qquad &  \stackrel{\mbox{\tiny{{\bf Thm 3.2}}}}{\longrightarrow }           
          & \qquad z_* =  (x_*, (\lambda \wedge 1) \nu_*, \lambda \eta_*)
 	\end{array}
        \]
 	\caption{Interchange of Limits Diagram}
 	\label{diag-comm}
            \end{figure}

            \noindent 
                {\em  Along with the tightness of $(\bar{Z}^{(N)}_*)_{N \in \N}$ established in \cite{KanRam12}, 
                  Theorem \ref{th-main2} of the present article completes the diagram by establishing (for a class of service distributions) 
            the right vertical  arrow (as explained in Remark
            \ref{rem-fix})  as well as the bottom horizontal arrow, though the latter only 
              when $\lambda \neq 1$ (i.e., in the
             subcritical and supercritical regimes). 
            It would be worthwhile in the future to  investigate
            whether this result can be extended further, in particular to establish
            converence even in the critical regime $\lambda = 1$, possibly under additional
            conditions such as a finite second moment condition, like that imposed in
            Theorem 3.9  of  \cite{KasRam11} (to study large-time behavior of fluid limits 
            in the absence of reneging). 
         }
         \end{remark}

 \section{Proof of Theorem \ref{th-main2}}
 \label{sec-proofs}
\beginsec

 We assume throughout this section that Assumption \ref{ass-main} holds.
We then have the following elementary lemma.

\begin{lemma}
	\label{lem-reneg}
        Fix $\lambda \geq 0$ and, 
 	given any $\freneg_0 \in \calM_F[0,H^r)$, let $\freneg = (\freneg_t)_{t \geq 0}$ be
	the solution to \eqref{eq-ftreneg}.    Then $\freneg_t \Rightarrow \lambda \frenegs$
	as  $t \rightarrow \infty$. 
\end{lemma}
\begin{proof}
   Fix   $\fsp \in {\mathcal C}_b(\R_+)$. 
   In view of \eqref{f4}, the boundedness of $\fsp$,   the
      finiteness of the measure $\eta_0$ 
     the dominated convergence theorem  and 
     the fact that  $\bar{G}^r(x+t)/\bar{G}^r(x) \rightarrow 0$ for every $x \in [0,H^r)$ as $t \rightarrow \infty$, together
       imply 
     that the first term  on  the right-hand side of
    \eqref{f4} vanishes.  On the other hand, since the mean patience time 
    $\int_0^\infty  \bar{G}^r(s) ds$ is finite, the dominated convergence theorem shows that the last
    term on the right-hand side of
    \eqref{f4} converges to $\langle \fsp, \lambda \eta_*\rangle$.
    This  concludes  the proof that $\freneg_t \Rightarrow \lambda \eta_*$ as  $t \rightarrow \infty$.  
      \end{proof}

\subsection{Proof in the Subcritical Regime}
\label{subs-pf2subcrit}

In this section we prove part (1) of Theorem \ref{th-main2}. 
Fix $\lambda \in (0,1)$ and
$(\fx(0), \fmeas_0, \freneg_0) \in \newcspace$.
Suppose $(\fx,\fmeas, \freneg)$ is a solution to the fluid equations, and let $(D, K, R, S, Q, B)$
be the corresponding auxiliary processes.

 The weak convergence of $\eta_t$ to $\lambda \frenegs$ as $t \to \iy$ follows from Lemma \ref{lem-reneg}.
We now analyze the remaining components of the solution.   
Using the definition of $D$ from \eqref{cond-radon}, setting $\fsp=h^s$ in (\ref{f5}),  interchanging the order of integration,  using integration by parts and the fact $G^s(0+) = 0$, we obtain 
\begin{eqnarray}
D(t) =   \ds \int_0^t \lan  h^s, \fmeas_s \ran \, ds
  &=&  \int_0^t \left(\int_{[0,H^s)}  \frac{ g^s(x+s)}{\bar{G}^s (x)} \fmeas_0 (dx)
	+ \int_{[0,s]} g^s(s-u) d\fk(u)\right)ds \\ &=&  \int_{[0,H^s)}  \frac{ G^s(t+x)-G^s(x)}{\bar{G}^s (x)} \fmeas_0 (dx) +\int_{[0,t]} G^s(t-u) d\fk(u) \nonumber \\ &=&  \int_{[0,H^s)}  \frac{ G^s(t+x)-G^s(x)}{\bar{G}^s (x)} \fmeas_0 (dx) + \int_0^t K(s) g^s(t-s)ds. \nonumber 
\end{eqnarray}
Substituting this in (\ref{eq-fx}),  using 
(\ref{qt-conserve})  and (\ref{fr}), and performing repeated integration by parts, we obtain 
\begin{eqnarray}
	\fx (t) &  = & \fx (0) + \lambda t  - \int_{[0,H^s)}  \frac{ G^s(t+x)-G^s(x)}{\bar{G}^s (x)} \fmeas_0 (dx)- \int_0^t K(s) g^s(t-s)ds - R(t)
	 \label{rep-X}
	\\ &=& 
	\fx (0) + \lambda t  - \int_{[0,H^s)}  \frac{ G^s(t+x)-G^s(x)}{\bar{G}^s (x)} \fmeas_0 (dx) \nonumber \\ & & \qquad - \int_0^t (Q(0)+\lambda s -Q(s)-R(s)) g^s(t-s)ds - R(t)  \nonumber \\ &=&  \fx (0) -Q(0)G^s(t)  - \int_{[0,H^s)}  \frac{ G^s(t+x)-G^s(x)}{\bar{G}^s (x)} \fmeas_0 (dx) + \int_0^t Q(s)g^s(t-s)ds   \nonumber \\ & & \qquad +  \lambda \int_0^t \bar G^s(t-s)ds
            +\int_0^t R(s)g^s(t-s)ds - R(t)  \nonumber \\ &=&   \fx (0) -Q(0)G^s(t)  - \int_{[0,H^s)}  \frac{ G^s(t+x)-G^s(x)}{\bar{G}^s (x)} \fmeas_0 (dx) +\int_0^t Q(s)g^s(t-s)ds \nonumber \\ & & \qquad +  \int_0^t \left(\lambda - \int_0^{\fq(s)}h^r((\fnsf)^{-1}(y))dy\right) \bar G^s(t-s)  ds,  \nonumber
\end{eqnarray}
which implies that for each $t\geq 0$, \begin{equation}X(t) \leq  \fx (0) -Q(0)G^s(t)  - \int_{[0,H^s)}  \frac{ G^s(t+x)-G^s(x)}{\bar{G}^s (x)} \fmeas_0 (dx) +\lambda \int_0^t \bar G^s(u)du +\int_0^t Q(s)g^s(t-s)ds. \label{Xest}\end{equation} 

We now make use of the following simple observation.  

\begin{lemma} \label{lem:limsupQ} 
	$\limsup_{t\rightarrow \infty} \int_0^t Q(s)g^s(t-s)ds \leq  \limsup_{t\rightarrow \infty}  Q(t).$
\end{lemma}
\begin{proof} Let $q:=\limsup_{t\rightarrow \infty}  Q(t)$. Then for each $\epsilon>0$, there exists $T_\epsilon < \infty$ such that $Q(t) \leq q+\epsilon$ for all $t\geq T_\epsilon$. So for each $t>T_\epsilon$,  it follows that 
	\begin{eqnarray*}\int_0^t Q(s)g^s(t-s)ds &=&   \int_0^{T_\epsilon}  Q(s)g^s(t-s)ds + \int_{T_\epsilon}^t Q(s)g^s(t-s)ds \\ &\leq & \left(\sup_{0\leq s\leq T_\epsilon} Q(s)\right) (G^s(t)-G^s(t-T_\epsilon)) + (q+\epsilon)G^s(t-T_\epsilon). \end{eqnarray*}
	By taking the limit supremum as $t \rightarrow \infty$ of both sides,
	we have $\limsup_{t\rightarrow \infty} \int_0^t Q(s)g^s(t-s)ds \leq q+\epsilon$.
	The lemma follows on taking $\epsilon\rightarrow 0$.
\end{proof}

Continuing with the proof of Theorem \ref{th-main2}(1), taking the limit supremum in  (\ref{Xest}), and using 
Lemma \ref{lem:limsupQ},
the identity $\int_0^\infty G^s(u) du = 1$ from Assumption \ref{ass-main},  the fact that 
$\lim_{t \rightarrow \infty} (G^s(t+x) - G^s(x))/\bar{G}^s(x) \rightarrow 1$ for every $x$,
the bounded convergence theorem and the identity
$X(0) = Q(0) + \langle \f1, \nu_0\rangle$ from \eqref{fq},    we obtain 
\begin{equation}
  \label{eq-temp}
  \limsup_{t\rightarrow \infty}  X(t) \leq \lambda + \limsup_{t\rightarrow \infty} \int_0^t Q(s)g^s(t-s)ds \leq \lambda + \limsup_{t\rightarrow \infty} Q(t).
\end{equation}

We now claim that there exists $T' <\infty$ such that $\lan  1, \fmeas_t \ran < 1$ for all $t \geq T'$.
We argue by contradiction to prove the claim. If the claim is false, note that  for any $T' <\infty$, there would exist
$T >T'$ such that $\lan  1, \fmeas_T \ran = 1$.  Then, due to (\ref{fqfx}), we would
have $\limsup_{t\rightarrow \infty}  X(t)  = \limsup_{t\rightarrow \infty}  Q(t) + 1$, which contradicts \eqref{eq-temp} since $\lambda <1$.
Thus, fix $T' < \infty$ as in the claim.  Then, 
 by Lemma \ref{lem:shift}, $(X^{[T']}, \nu^{[T']}, \eta^{[T']})$ solves the fluid equations with arrival rate $\lambda$ and initial condition 
  $(X(T'), \nu_{T'} ,\eta_{T'})$ 
  and hence, \eqref{f5} holds with $\nu$ and $K$ replaced with $\nu^{[T']}$ and $K^{[T']}$, respectively.
Since 
$\nu^{T'}(t) = \nu_{T'+t}$  and 
      by  \eqref{fqfx}, \eqref{fr} and  \eqref{qt-conserve},
      $Q(T'+ \cdot) \equiv  0$, $R^{[T']}(\cdot) \equiv 0$ and $K^{T'} (t)  = K(T'+t) - K(T') = \la t,$ $t \geq 0$, 
this implies that   for  every $\fsp \in {\mathcal C}_b[0,H^s)$, 
  \[  \int_{[0,H^s)} \fsp (x) \nu_{T' + t} (dx) = \int_{[0,H^s)} \fsp (x+t) \frac{\bar{G}^s(x+t)}{\bar{G}^s(x)}  \nu_{T'} (dx)
      + \int_0^t \fsp (t-s) \bar{G}^s  (t-s) \lambda ds. 
      \]
      Then, arguing as in the proof  of Lemma \ref{lem-reneg},  
      sending $t \rightarrow \infty$, and invoking the bounded convergence theorem, the first integral
      on the right-hand side  vanishes, and the second integral converges to $\lambda \int_{[0,H^s)} \fsp(x) \bar{G}^s(x) dx$.  
      Recalling  that $\nu_* (dx) = \bar{G}^s(x) dx$ and $\int_0^\infty \bar{G}^sds = 1$ from
      Assumption \ref{ass-main}, it  follows  that 
      $\nu_t \Rightarrow \la \nu_*$.  In turn, by the continuous
      mapping theorem this implies 
      $\langle \f1, \nu_t \rangle \Rightarrow \lambda$ as
      $t \rightarrow \infty$. 
      When combined with \eqref{eq-fnonidling} and the fact that  $\lambda < 1$, this implies  that as $t \rightarrow \infty$,
      the weak  limits of $X(t)$ and $\langle \f1, \nu_t \rangle$  coincide and are equal to $\lambda$.
      This concludes the proof of the  first  assertion of Theorem \ref{th-main2}(1).
      
    Now, if the initial condition $(X(0), \freneg_0, \fmeas_0)$
      had the law $\mu$ of a random fixed point with arrival rate $\lambda < 1$, then the convergence
      just established would imply that $\PP(\freneg_0 = \lambda \frenegs) = 1$ and $\PP(\fmeas_0 = \lambda \fmeass) = 1$.
      By the continuous mapping theorem, the latter implies that almost surely $\langle \f1, \fmeas_0\rangle = \langle \f1, \lambda \fmeass\rangle
      = \lambda$.   Since $\lambda < 1$,  it then follows from \eqref{eq-fnonidling} that $X(0) = \lambda$ almost surely, thus
      proving that $\mu = \delta_{z_*^\lambda}$ with $z_*^\lambda = (\lambda, \lambda \frenegs, \lambda \fmeass)$. 
This completes  the proof of Theorem \ref{th-main2}(1).

\subsection{Proof of Theorem \ref{th-main2}(2) when the hazard rate function is bounded away from
zero and infinity.}
\label{subs-pf1}

In this section we prove Theorem \ref{th-main2}(2a) under  Assumption \ref{ass:a}(\ref{ass:a1}).  
Fix  $\lambda > 1$,  suppose Assumption \ref{ass-main} is satisfied, Assumption \ref{ass-unique} holds (with $z_*^\lambda   = (x_*^\lambda, \nu_*,  \lambda \nu_*)$ denoting the  unique element of ${\mathcal I}_\lambda$)  and 
 Assumption \ref{ass:a}(\ref{ass:a1}) holds (with associated positive constants $\eps_h > 0$, $c_h < \infty$). 
For notational  convenience, we shall denote by $f^*(x) = \bar{G}^s(x)$ the density of $\nu^*$.
Note that the lower bound on $h^s$ implies that $g^s$, and thus $f^* = \bar{G}^s$, 
is strictly positive on $(0,\infty)$.

Now, fix  the initial condition $(X(0), \nu_0, \eta_0) \in \newcspace$,
and suppose $(X,\nu, \eta)$ is the associated  solution to the fluid  equations. 
We will establish convergence, as $t \rightarrow \infty$,  of the fluid age measure $\nu_t$ described by \eqref{eq-ftmeas} 
using an extended relative entropy functional in a manner reminiscent of a Lyapunov function.
Recall that  ${\mathcal P}(E)$ denotes the space of probability measures on  a measurable space $E$, and 
for a {\em finite} measure $P$ on  $E$, define the functional $R:(P\|\cdot): {\mathcal P}(E) \mapsto (-\infty, \infty]$  by
\begin{equation}
  \label{eq-Rfnal}
R(P\|Q) := \left\{
\begin{array}{ll}
 \displaystyle \int_{E} \log\frac{dP}{dQ}(x) dP(x) & \mbox{ if }  P\ll Q,\\
  \infty & \mbox{ otherwise, }
  \end{array}
\right.
\end{equation}
where $P \ll Q$ means $P$ is absolutely continuous with respect to $Q$ and we use the convention $0 \log 0 = 0$. 
We emphasize that we do not require $P$ to be a probability measure, as we will often have to deal with sub-probability measures,
but when both $P$ and $Q$ are probability measures, this is simply the relative entropy functional. 

\begin{remark}
  \label{rem-relent}
        {\em
          If $c_P = P(E) > 0$ denotes the total mass of $P$, then writing the above integral as $\int_E \frac{dP}{dQ}\log\frac{dP}{dQ}dQ$
and using the convexity of $x\mapsto x\log x$ on $(0,\infty)$ gives the lower bound 
\begin{equation}\label{10}
R(P\|Q)\ge c_P\log c_P, 
\end{equation}
which  is attained by $P$ that is a constant 
multiple of the probability measure $Q$. In particular,
$R(P\|Q)$ may assume negative values.
However, when $P$ is a probability measure, $R(P\|Q)$ is always nonnegative 
and $R(P\|Q)=0$ holds if and only if $P=Q$.
}
\end{remark}

 The proof of Theorem \ref{th-main2}(2) will make use of the following properties of the extended relative 
  entropy functional. 
 
\begin{lemma}
\label{lem-Pinsker}
Suppose $P$ and $Q$ are finite measures on $\R_+$, equipped with the Borel $\sigma$-algebra,
with $c_P := P(\R_+) > 0$ and $Q(\R_+) = 1$. 
If $P$ and $Q$, respectively, have densities  $p$ and $q$ (with respect to Lebesgue measure),
then 
\begin{equation}\label{05}
\int_0^\iy|p(x)-q(x)|dx
\le  |c_P-1| +\Big(2c_P^{-1}|R(P\|Q)|+2|\log c_P|\Big)^{1/2}. 
\end{equation} 
\end{lemma}
\begin{proof}
  First note that $c_P^{-1} P$ and $Q$ are probability measures, and so, invoking
  Pinsker's inequality (see, e.g., \cite{CKbook}, p.\ 44)
in the second  inequality below, we obtain 
\begin{eqnarray*} \int_0^\iy|p(x)-q(x)|dx  & \leq  & 
\int_0^\iy|p(x)-c_P^{-1}p(x)|dx  + \int_0^\iy|c_P^{-1}p(x)- q(x)|dx  \\
& \leq & |c_P-1|  +  \left( 2 R(c_P^{-1} P \| Q) \right)^{1/2} \\
& \leq & |c_P-1| + \left( 2 c_P^{-1}  R(P \| Q ) - 2 \log c_P \right)^{1/2}
\end{eqnarray*}
which is clearly dominated by the right-hand side of \eqref{05}.
\end{proof}

The second property is encapsulated in the following lemma,
which crucially relies on the lower bound on the hazard rate $h^s$,
and whose proof is relegated to Appendix \ref{sec-lem1}.

\begin{lemma}\label{lem1}
  Let $f: [0,\infty) \mapsto [0,\infty)$ be a  measurable function that satisfies $\int_0^\iy f dx\le 1$, 
 suppose $z_f:=\int_0^\iy h^s f dx<\iy$ and $\mu^f$ is the measure with density $f$.  Then 
\begin{equation}
  \label{estimate}
\int_0^\iy h^s(x) f(x)\log\frac{f(x)}{f^*(x)}dx-z_f \log z_f
\ge \eps_h \int_0^\iy f(x)\log\frac{f(x)}{f^*(x)} dx =  \eps_h R(\mu^{f}\|\nu_*). 
\end{equation}
\end{lemma}

The proof of Theorem \ref{th-main2}(2) is somewhat involved and given in Section \ref{subsub-pfsuper}.  
To help make some of those calculations more transparent, first in Section \ref{subsub-formal} we carry
out some  formal calculations (under more stringent conditions)
to provide intuition into why the extended relative entropy functional  $R(\cdot\|\nu^*)$ may be a good candidate  
Lyapunov function for the problem at hand (see also Remark \ref{rem-pde}).

\subsubsection{A Formal Calculation}
\label{subsub-formal}

 Observe that  equation \eqref{eq-ftmeas} characterizes  $(\nu_t)_{t \geq 0}$ as a weak solution to a transport equation.
Now,  for the purposes of this formal calculation only, suppose that $\nu_0$ has a density,
  denoted by $\tpsi_0$, and for each $t > 0$, suppose the measure $\nu_t$ has a sufficiently   smooth
  density,   denoted by $\tpsi(x,t)$, $x\geq 0$.   For conciseness, below
  we will use $f(\cdot, t)$ to denote the function $x \mapsto f(x,t)$. 
  Then  by \eqref{def-q} and \eqref{cond-radon},
  $\langle \f1, \nu_t\rangle 
   = \int_0^\infty \tpsi(\cdot,t) dx$   and $\langle h, \nu_t \rangle = \int_0^\infty h^s \tpsi(\cdot,t) dx$, 
the transport equation could be formally rewritten as the following 
partial differential equation (PDE):   
\begin{equation}\label{11}
\pl_t \tpsi(x,t)=-\pl_x \tpsi(x,t)-h^s(x)\tpsi(x,t), \quad x >0, t >0, 
\end{equation}
with the boundary condition $\tpsi(0,t) = K^\prime(t)$, which by \eqref{fkprime}, takes the form 
\begin{equation}\label{12}
\tpsi(0,t)=\begin{cases}
\la & \text{if }  \displaystyle \int_0^\infty \tpsi (\cdot,t)dx<1,\\[0.9em]
\displaystyle \int_0^\infty h^s\tpsi(\cdot,t)dx & \text{if } \displaystyle  \int_0^\infty \tpsi (\cdot,t)dx=1,
\end{cases}
\end{equation}
and the initial condition
\begin{equation}\label{13}
\tpsi (x,0)=\tpsi_0(x), \quad x  > 0. 
\end{equation}
Proceeding with  purely formal calculations to gain intuition,  note that  $f^*=e^{-J}$, where 
$J(x) :=\int_0^xh^s(y)dy < \infty$ for every $x > 0$. 
For $t > 0$, define 
\[
r_t := R(\nu_t \|\nu^*) 
= \int_0^\iy \tpsi(\cdot, t)\log\frac{\tpsi(\cdot, t)}{f^*}dx=\int_0^\iy \tpsi (\cdot, t)(\log \tpsi(\cdot, t)+J)dx.
\]
Taking derivatives of both sides of the last equation with respect to $t$, and  using \eqref{11}, we see that 
\begin{align*}
\frac{d}{dt}r_t
&= \int_0^\iy \partial_t \tpsi (\cdot, t) (\log \tpsi(\cdot, t)+J +1)dx
\\
&=-\int_0^\iy(\partial_x \tpsi(\cdot, t) +h^s \tpsi(\cdot, t))(\log \tpsi(\cdot, t)+J +1)dx.
\end{align*}
Since $\tpsi(\cdot,t)$ is integrable and $H^s =\infty$, it follows that $\liminf_{x \rightarrow \infty} \tpsi(x,t) = 0$. 
Using integration by parts, and assuming (without justification)  that 
$\lim_{x \rightarrow \infty} \tpsi(x,t)(\log \tpsi(x,t)+J(x)) = 0$, we conclude that 
\begin{align*} 
  \int_0^\iy(\partial_x \tpsi(\cdot,t) (\log \tpsi(\cdot,t)+J +1)dx 
& =   -\tpsi(0,t)(\log \tpsi(0,t)+1)-\int_0^\iy \tpsi(\cdot,t)\left(\frac{\partial_x \tpsi(\cdot,t)}{\tpsi(\cdot,t)}+h^s\right)dx\\
&= -\tpsi(0,t)\log \tpsi(0,t)-\int_0^\iy h^s \tpsi(\cdot,t)dx,
\end{align*}
On combining  the last   two equations, and recalling that $J = -\log f^*$, we obtain 
\begin{align*}
  \frac{d}{dt}r_t
 &= \tpsi(0,t)\log \tpsi(0,t)-\int_0^\iy h^s \tpsi(\cdot,t)(\log \tpsi(\cdot,t)+J)dx\\
&= \tpsi(0,t)\log \tpsi(0,t)-\int_0^\iy h^s\tpsi(\cdot,t)\log\frac{\tpsi(\cdot,t)}{f^*}dx.
\end{align*}
Since $\int_0^\infty \tpsi(\cdot,t) dx   =  \langle \f1, \nu_t \rangle \leq 1$ and for almost every $t \in [0,\infty)$,
  \eqref{cond-radon} implies that $\int_0^\infty h f(\cdot, t) dx <  \infty$  for such $t$. 
we can apply the estimate \eqref{estimate} from Lemma \ref{lem1} with $f  = \tpsi(\cdot,t)$ to obtain 
\[
\int_0^\iy h^s \tpsi(\cdot,t)\log\frac{\tpsi(\cdot,t)}{f^*}dx\ge \left(\int_0^\iy h^s\tpsi(\cdot,t)dx\right)\log\left(\int_0^\iy h^s\tpsi(\cdot,t)dx\right)
+\eps_h\int_0^\iy \tpsi(\cdot,t)\log\frac{\tpsi(\cdot,t)}{f^*}dx.
\] 
Substituting this into the previous display and using the boundary condition \eqref{12}, we  have 
\begin{equation}\label{15}
\frac{d}{dt}r_t\le\begin{cases}
-\eps_hr_t+\la\log\la-\left(\int_0^\iy h^s\tpsi(\cdot,t) dx\right)\log\left(\int_0^\iy h^s\tpsi(\cdot,t) dx\right) & \text{if } \int_0^\iy \tpsi(\cdot,t) dx<1,\\
-\eps_hr_t & \text{if } \int_0^\iy \tpsi(\cdot,t) dx=1.
\end{cases}
\end{equation}
This estimate does not directly imply the convergence of $r_t$ to zero.
However, the fact that it takes the form $\frac{dr_t}{dt}\le -\eps_hr_t$ in the  case $\int_0^\iy \tpsi(\cdot,t) dx=1$ 
is a sign that the approach might be useful, especially in the supercritical case ($\lambda > 1$), where
one might expect that for sufficiently large $t$,  $\int_0^\infty f(\cdot,t)dx  = 1$.
However, translating this intuition into a proof is not straightforward. 
The rigorous argument provided in the next section indeed derives a version of \eqref{15}
(with some extra error terms), and copes with the more complicated
structure of the estimate in the case $\int_0^\iy \tpsi (\cdot, t) dx<1$, as well as the fact that $r_t$
can go negative.

\begin{remark}
  \label{rem-pde}
  {\em 
Note that the PDE  \eqref{11}-\eqref{12} has some similarities with the age-structured model in equation (3) of \cite{MicMisPer04}, with
$\nu = 0$ and $d = b= h^s$, except that the boundary condition  \eqref{12} is   more complicated.  In particular, it is
discontinuous due to the appearance of the  term $\la$ when  $\int_0^\infty \tpsi (\cdot,t)dx<1$.   Furthermore,
although
$f^*$ can indeed be seen as an eigenfunction corresponding to the eigenvalue $0$ of the stationary equation
(which corresponds to equation (7) of \cite{MicMisPer04}, 
 again with $\nu =  0$ and $d = b= h^s$), 
 since the hazard rate function $h^s$ is never integrable on $[0,\infty)$,  the 
   solution to the dual equation (see (8) of \cite{MicMisPer04}) appears  not to be well-defined.
   Thus,    the results of \cite{MicMisPer04} are not  applicable to this setting. 
Furthermore, a rigorous proof cannot in any case rely on an analysis of the PDE because
for general initial condition $\nu_0 \in {\mathcal M}_F[0,\infty)$, the measures $\nu_t, t > 0,$ need not have
densities, and even when they do, their densities have discontinuities in both variables
(these discontinuities will be apparent in the rigorous proof in  the  next section). 
Nevertheless, along with the calculations given above, this loose analogy  further suggests that  the extended relative entropy functional
   may still serve as a Lyapunov function for the dynamics.  That verification of this property is non-trivial
   will be apparent on noting that  it requires additional conditions on $h^s$ and also a
   restriction to the supercritical regime $\lambda > 1$.  In particular, it would be interesting
   to see if the  argument presented in the next section,  or a modification thereof, could
relax conditions on $h^s$ to address a larger class of service distributions, 
and  also    address the critical regime $\lambda = 1$, which currently we only address when the
hazard rate function is decreasing (see Proposition \ref{prop-key}). 
}
\end{remark}

\subsubsection{Proof of Theorem \ref{th-main2}(2)}
\label{subsub-pfsuper}

Fix $\lambda> 1$ and recall the initial condition and associated solution $(X,\nu,\eta)$ to the fluid equations. 
We start with the proof of part (a).   First, note that the  limit $\eta_t \Rightarrow \lambda \eta_*$
in \eqref{conv-1} follows from Lemma \ref{lem-reneg}. 
To establish the remaining limits, we begin with the  representation for the age measure $\nu_t$ given in
\eqref{f5}, which shows that  $\nu_t = \newdel_t + \mu_t$, where  
$\newdel_t, \mu_t \in {\mathcal M}_F[0,\infty)$, are defined by 
\begin{equation}\label{01} 
  \langle \fsp, \newdel_t \rangle := \int_{[0,\iy)}\frac{\bar G^s(x+t)}{\bar G^s(x)}
    \fsp (x+t)\nu_0(dx) \quad \mbox{ and } \quad 
    \langle \fsp, \mu_t \rangle :=  \int_0^\infty \fsp(x) \ttpsi(x,t) dx, 
\end{equation} 
for every $\fsp \in {\mathcal C}_b[0,\infty)$ and $\fsp  = h^s$,  where for all $t \geq 0$,
    \begin{equation}\label{14-}
\ttpsi(x,t) :=\begin{cases}\bar G^s(x)k_{t-x} & x\in[0,t],
\\ 0 & x\in(t,\iy), 
\end{cases}
\end{equation}
where we recall that $k$, defined in \eqref{fkprime}, is a.e.\ equal to the derivative $K'$ of $K$.

Now, to estimate $d_{\text{TV}} (\mu_t, \nu_*)$, 
recall that  $f^*=\bar G^s$ is the density of $\nu^*$, and so  
 both $\mu_t$ and $\nu_*$ are absolutely continuous with respect to Lebesgue measure. 
Thus, Lemma \ref{lem-Pinsker} shows that
\begin{equation}
  \label{ineq-Pinsker}
  d_{{\rm TV}} (\mu_t, \nu_*) = \int_0^\infty \left| \ttpsi(x,t) - f^*(x) \right| dx \leq
|\langle \f1, \mu_t \rangle - 1|
+ \left(2\langle \f1, \mu_t \rangle^{-1}|r_t| + 2|\log \langle \f1, \mu_t \rangle|\right)^{1/2},
\end{equation}
where,  for $t \geq 0$, 
\begin{eqnarray}
  \label{def-rt}
  r_t := R(\mu_t\|\nu^*) &= & \int_0^\iy \ttpsi(x,t)\log\frac{\ttpsi(x,t)}{f^*(x)}dx \\
   \label{def-rt2}
&=&\int_0^t\bar G^s(t-x)k_{x}\log k_{x}dx, 
\end{eqnarray}
with the last equality using the fact that $k_{t-x} = \ttpsi(x,t)/f^*(x)$ due to 
\eqref{14-}.   Since the expression  in \eqref{fkprime} and Assumption \ref{ass:a}(\ref{ass:a1})
show that $k$ is strictly
positive and bounded above by $\lambda \vee c_h = \lambda\vee\sup_{x \in [0,\infty)} h(x)$, $r_t$ is well defined and finite.

 \begin{remark}
   \label{rem-theta}
 {\em  Due to the pointwise convergence $\frac{\bar G^s(x+t)}{\bar G^s(x)}\to0$ as $t\to\iy$,
the dominated convergence theorem shows that $\langle \f1, \newdel_t\rangle$, the total mass of $\newdel_t$, converges to zero
as $t\to\iy$. Hence, $\newdel_t$ converges to the zero measure in total variation. 
Together with \eqref{ineq-Pinsker}, it follows that in order to show 
$B_t = \langle \f1, \nu_t \rangle \to 1$ and 
$d_{\text{TV}} (\nu_t, \nu_*) \to 0$ (and hence, $\nu_t \Rightarrow \nu_*$)  as $t \to \iy$, 
it suffices to prove that  $\langle \f1, \mu_t \rangle \to 1$ and $r_t \to 0$ as $t \rightarrow \infty$. 
 }
 \end{remark}

Our main goal in this section is  to establish these limits. 

    \begin{proposition}
      \label{prop-superpf} 
      Suppose Assumptions \ref{ass-main} and \ref{ass:a}(\ref{ass:a1}) hold, and $\lambda > 1$.
      Then there exists $T \in (0, \infty)$ such that $B(t) = 1$ for all $t \geq T$.
      In addition, 
\begin{equation}\label{14}
\langle \f1, \mu_t \rangle \to 1 \quad \text{ and } \quad r_t\to 0, \quad \mbox{ as } t \to \infty,
\end{equation}
and also $\langle h, \nu_t \rangle \rightarrow 1$ as  $t \rightarrow \infty$. 
    \end{proposition}

    To establish this proposition, we proceed in several steps, establishing various intermediate results 
    in Steps 1--3,  culminating in the proof of Proposition \ref{prop-superpf} in Step 4.   

  \paragraph{{\bf Step 1.}}  We start with simple bounds on the measure $\theta_t$ defined in  \eqref{01}. Recall  the convention $0\log0=0$. 

\begin{lemma}\label{lem2} We have 
$\sup_t \langle h^s, \newdel_t \rangle \le c_h$,
$\int_0^\iy\lan h^s,\newdel_t\ran dt<\iy$, and
$\int_0^\iy|\lan h^s,\newdel_t\ran\log\lan h^s,\newdel_t\ran|dt<\iy$.
\end{lemma}
\begin{proof}
Recall that $\newdel_t = \nu_t - \mu_t$ is a nonnegative measure. Moreover,
substituting $\fsp = h^s$ in \eqref{01},  we have for each $t > 0$, 
\begin{align}\label{121}
\lan h^s,\newdel_t\ran&=\int_{[0,\iy)}\frac{\bar G^s(x+t)h^s(x+t)}{\bar G^s(x)}\nu_0(dx).
\end{align}
For the first assertion, note that for all $t \geq 0$, 
$\langle h^s, \newdel_t   \rangle \leq c_h \langle \f1, \nu_t   \rangle \leq  c_h$.
The remaining claims will follow once we prove the following refinement of
this bound, namely, for all $t\ge0$,
\begin{equation}\label{122}
\lan h^s,\newdel_t\ran\le c_he^{-\eps_h t}.
\end{equation}
To see why this bound holds, first use the easily verifiable relation $\bar G^s (y) = e^{-\int_0^y h^s(u) du}$ and
the definition of $\eps_h$ to conclude that 
for all $x\ge0$ and $t\ge0$,
$\bar G^s(x+t) \le\bar G^s(x)e^{-\eps_ht}$.
When substituted into \eqref{121}, this yields 
\[
\lan h^s,\newdel_t\ran \le e^{-\eps_ht}\int_{[0,\iy)}h^s(x+t)\nu_0(dx)
\le c_he^{-\eps_ht}\lan\f1,\nu_0\ran\le c_he^{-\eps_ht}. 
\]
This proves \eqref{122} and completes the proof.
\end{proof}

\paragraph{{\bf Step 2.}}
We now obtain our main estimate on $r_t$ in Corollary \ref{cor-lem3}, building off preliminary estimates
obtained in Lemma \ref{lem-rtprops}. 
In what follows, we will say   $(t_1,t_2) \subset [0,\infty)$ is a {\it busy interval} if $B_t=1$ for $t \in (t_1,t_2)$,  and say it is an {\it excursion interval} if $B_t<1$ for $t \in (t_1,t_2)$ and $B_{t_1}=B_{t_2} =1$.

Let $\moc (\cdot)$ denote the modulus of continuity of
the continuous function $x \mapsto x\log x$  on the compact interval
$[0,c_h]$. 
On $[0,e^{-1}]$ this function is decreasing. Now, for $0\le x<y\le e^{-1}$,
applying the inequality $p\log p+(1-p)\log(1-p)\le0$ with  $p=x/y$, we see that 
\begin{align*}
 0  &\geq \frac{x}{y}\log\frac{x}{y}+\left(1-\frac{x}{y}\right)\log\left(1-\frac{x}{y} \right)
 =\frac{1}{y}[x\log x+(y-x)\log(y-x)-y\log y]. 
\end{align*}
Hence, it follows that for $0\le x<y\le e^{-1}$, 
\[
|x\log x-y\log y|=x\log x-y\log y\le(x-y)\log(y-x)
=|(x-y)\log(y-x)|.
\]
Moreover, in case $c_h>e^{-1}$,
the  function  $x \mapsto x\log x$  is Lipschitz on $[e^{-1},c_h]$. 
As a result, there is a constant $c_1$ (depending only
on $c_h$) such that
\begin{equation}\label{120}
\moc (x)\le |x\log x|+c_1x, \qquad x\in[0,c_h].
\end{equation}

\begin{lemma}
  \label{lem-rtprops}
     For $t \geq 0$, define $\neweta_t := \moc(\langle h^s, \newdel_t \rangle),$
where $\newdel_t$ is defined by \eqref{01}, and $\moc$ is the modulus of continuity of $x \mapsto x \log x$, as defined above. 
    If $(t_1,t_2)$ is a   busy interval, then 
\begin{equation}\label{07}
r_t\le r_{t_1} e^{-\eps_h(t-t_1)}+\int_{t_1}^t \neweta_s ds, \quad t\in (t_1, t_2).
\end{equation}
On the other hand, if $(t_1,t_2)$ is  an excursion, then 
\begin{equation}\label{08}
r_{t_2} \le r_{t_1} +\int_{t_1}^{t_2}\neweta_sds,
\end{equation}
and 
\begin{equation}\label{1:1}
B^\prime (t) =\la-\lan h^s,\nu_t\ran, \qquad t\in(t_1,t_2).
\end{equation} 
Furthermore,  there exist  finite positive  constants $c_r$ and $c_{\rm lip}$ such that
$\sup_t |r_t| \leq c_r$ and  for any $0 \leq s < t < \infty$, $|r_t - r_s| \leq c_{\rm lip}|t-s|$, showing that the function 
$t \to r_t$ is globally Lipschitz on $[0,\infty)$.  
\end{lemma}
\begin{proof}
  Note that 
  although the function $\ttpsi$  defined in \eqref{14-} is discontinuous in $t$ and in $x$,
  since $\bar{G}^s$ has a density, 
  the relation   \eqref{def-rt2} shows that 
$r_t$ is differentiable (although not continuously differentiable)  with derivative  
\begin{eqnarray}\label{03}
  \frac{dr_t}{dt} & = & k_t\log k_t-\int_0^tg^s(t-x)k_x\log k_xdx = k_t\log k_t-\int_0^tg^s(x)k_{t-x}\log k_{t-x}dx. 
\end{eqnarray}
Substituting the identities  $g^s= h^s \bar G^s = h^s f^*$ and
$k_{t-x} =  \ttpsi(x,t)/f^*(x)$ into \eqref{03},  recalling the definition of
$\ttpsi$ from \eqref{14-} recalling the convention that
$0 \log 0 = 0$, and then applying Lemma \ref{lem1}, with $f$ replaced with $\ttpsi(\cdot,t)$, we obtain 
\begin{align}
\notag
\frac{dr_t}{dt} &=k_t\log k_t-\int_0^\iy h^s(x)\ttpsi(x,t)\log\frac{\ttpsi(x,t)}{f^*(x)}dx\\ 
&\le k_t\log k_t-z_{\ttpsi(\cdot,t)}\log z_{\ttpsi(\cdot,t)}-\eps_h r_t^+,
\label{20}
\end{align}
where, as in Lemma \ref{lem1}, 
$z_{\ttpsi(\cdot,t)} = \int_0^\infty h^s(x) \ttpsi(x,t) dx$, which is  equal  to   $\langle h^s, \mu_t \rangle$  by \eqref{01}.

Now, suppose that  $(t_1,t_2)$ is a busy interval for some $0 \leq t_1 < t_2 \leq \infty$.  
Then by  \eqref{fkprime} and Assumption \ref{ass:a}(\ref{ass:a1}), for $t\in(t_1,t_2)$, 
$c_h \geq k_t=\lan h^s,\nu_t\ran = \langle h^s, \mu_t \rangle + \lan h^s,  \newdel_t \rangle$,  which implies 
$k_t-z_{\ttpsi(\cdot,t)}=\lan h^s,\newdel_t\ran\ge0$.
Since $\moc$ is the modulus of continuity of
$x\mapsto x\log x$ on the interval $[0,c_h]$,
it follows that
\[
|k_t\log k_t-z_{\ttpsi(\cdot,t)}\log z_{\ttpsi(\cdot,t)}|
\le  \moc (\langle h^s, \newdel_t \rangle) = \neweta_t. 
\]
When combined with \eqref{20}, this shows that
  for any  busy interval $(t_1,t_2)$, 
\begin{equation}\label{21}
  \frac{dr_t}{dt}\le \neweta_t-\eps_h r_t^+ \le  \neweta_t - \eps_h r_t, \quad t \in (t_1, t_2).
\end{equation}
Now, let $\tilde{r}$ denote the solution to the differential equation $d \tilde{r}_t/dt=\neweta_t-\eps_h\tilde{r}_t$
with the same initial condition as $r$, namely $\tilde{r}_{t_1}=r_{t_1}$.
Then $\tilde{r}$ can be solved explicitly: 
 \[
 \tilde{r}_t=\tilde{r}_{t_1}^{-\eps_h(t-t_1)}+\int_{t_1}^te^{-\eps_h (s-t_1)}\neweta_{t-s}ds\le
 r_{t_1} e^{-\eps_h(t-t_1)}+\int_{t_1}^t\neweta_sds,   \quad  t\in (t_1, t_2).
\]
A simple comparison theorem for ordinary differential equations then shows that $r_t \leq \tilde{r}_t$ for
$t \in (t_1,t_2)$. This proves \eqref{07}.

Next, consider an excursion interval $(t_1,t_2)$.
Then \eqref{fkprime} implies that $k_t=\la$ for $t \in (t_1,t_2)$.   
Moreover, it is not hard to see that the fluid age equation \eqref{eq-ftmeas} holds with the test function $\ph \equiv \f1$,
 by approximating this function by compactly supported test functions whose 
 derivatives in $x$ are bounded. Since $\varphi_x = \varphi_t = 0$, differentiating the equation yields \eqref{1:1}.
 Toward showing \eqref{08}, recall  that  $\langle \f1, \nu_t \rangle = B(t)$ and
 $B(t_1) = B(t_2) = 1$ by definition of an excursion interval. 
  Hence, it follows that 
\begin{equation}\label{123}
\frac{1}{t_2-t_1}\int_{t_1}^{t_2}\lan h^s,\nu_t\ran dt=\la. 
\end{equation}
Also, recalling $\newdel_t = \nu_t-\mu_t$, we have 
\[ |\lan h^s,\mu_t\ran\log\lan h^s,\mu_t\ran - \lan h^s,\nu_t\ran\log\lan h^s,\nu_t\ran|
\leq
\moc(\lan h^s,\nu_t\ran - \lan h^s,\mu_t\ran) =  \neweta_t.
\]
As a result, for $t\in(t_1,t_2)$, the right-hand side of \eqref{20} is bounded above
by $\la\log\la-\lan h^s,\nu_t\ran\log\lan h^s,\nu_t\ran+\neweta_t$.
Integrating both sides of \eqref{20} and using 
\eqref{123} and the convexity of $x \log x$, we have
\[
r_{t_2} - r_{t_1} \leq
(t_2-t_1)  \la\log\la - \int_{t_1}^{t_2} \left( \lan h^s,\nu_t\ran\log\lan h^s,\nu_t\ran -  \neweta_t \right) dt
\leq \int_{t_1}^{t_2} \neweta_t dt.
\]

We now turn to the last assertion of the lemma.
The bound $0\le k_t\le c_h \vee \lambda$ implies that
$|k_t\log k_t|\le c_2$ for some finite constant $c_2$.
The boundedness of $t \to r_t$ thus
follows from \eqref{def-rt2} and the fact that
$\int_0^\infty \bar{G}^s(x) dx  = 1$ (see Assumption \ref{ass-main}(1)).
By \eqref{03}, the bound on $|k_t\log k_t|$ also implies that $dr_t/dt$ is bounded, and hence, 
 that $t \mapsto r_t$  is globally Lipschitz on $[0,\infty)$.
    \end{proof}

 As a corollary, we  obtain our main estimate on $r_t$.  For $t > 0$, define 
\begin{equation}
    \label{def-Tt}
    \TT (t) :=\int_0^t1_{\{B(s)=1\}}ds, \quad t > 0, \quad \mbox{ and } \quad \calB :=\{t > 0:B(t)=1\}. 
\end{equation}

\begin{corollary}\label{cor-lem3}
   For every $s \geq 0$ and  $t>s$, $t\in\calB$, 
      \begin{equation}
        \label{ineq-lem3}
r_t\le c_re^{-\eps_h (\TT(t)-\TT(s))}+\int_s^t\neweta_\tau d\tau, 
      \end{equation}
      where $c_r$ is the  constant from Lemma \ref{lem-rtprops}. 
\end{corollary} 
\begin{proof}  
Fix $s \geq 0$ and $t > s$, $t \in\calB$.
Denote $t_0 :=\inf\{u \ge s:B_u =1\}$.
Fix a nonempty open interval $(s_0,s_1)\subset(t_0,t)$.
Then $(s_0,s_1)$ is said to be a {\it maximal} busy interval if it is a busy interval that is not a proper subset of
any open busy interval contained in $(t_0,t)$. 
Further,  $(s_0,s_1)\subset(t_0,t)$
is referred to as {\it admissible} if it is either an excursion
or a maximal busy interval.  
Since $B$ is continuous it is clear that ${\mathcal O} :=  \{u \in (s,t): B_u < 1\}$ is an open set, and hence, 
  can be written as a countable union of open intervals. Thus, there are at most a countable number of
  excursions.  Since any maximal busy interval must be contiguous to one of the intervals comprising ${\mathcal O}$,
  it follows that   the collection of admissible intervals is also countable. 
For $u>0$, define a {\it $u$-admissible} interval to be  an admissible interval
whose length is at least $u$. Denote by $\calT_u$ the complement in $(t_0,t)$ of the
union of all $u$-admissible intervals. Then, as $u\to0$,
the Lebesgue measure $|\calT_u|$ of this set clearly converges to zero.

Let $u>0$ be given, and let $ \numb_u$ be the number of $u$-admissible intervals. 
Since there are only a finite number of such intervals, we can label the intervals 
$(t_n,t_n')$, $n=1,\ldots, \numb_u$  in such a way
that $s\le t_0\le t_1<t_2<t_3<\cdots<t_{\numb_u}\le t$.
Let $c_{\rm lip}$ denote the (global) Lipschitz constant of $t \mapsto r_t$,
which exists by Lemma \ref{lem-rtprops}. 
We now show by induction that, for $n=1,2,\ldots,\numb_u$,
\begin{equation}\label{09}
r_{t_n}\le r_{t_1}e^{-\eps_h\sum_{i=1}^{ n-1}(\TT(t_i')-\TT(t_i))}
+\sum_{i=1}^{n-1}\int_{t_i}^{t_i'}\eta_\tau d\tau+c_{\rm lip}\sum_{i=1}^{n-1}(t_{i+1}-t'_i),
\end{equation}
 where a sum with the upper limit less than the lower limit is taken to be zero.

\noindent 
    {\em Base Case:} For $n=1$,  \eqref{09}  reduces to the trivial  inequality $r_{t_1} \leq r_{t_1}$, and thus is satisfied.

    \noindent 
{\em Induction step:} Assuming \eqref{09} holds 
for an arbitrary $n \in \{1, \ldots, \numb_u-1\}$, we show it holds for $n+1$.
From \eqref{07} and \eqref{08} of Lemma \ref{lem-rtprops},
along with the fact that  $\TT(t_n')-\TT(t_n)$  is equal to zero if  $(t_n,t_n')$ is an excursion, and is equal to $t_n'-t_n$ if it is a busy interval, we have
\[
r_{t_n'}\le r_{t_n}e^{-\eps_h(\TT(t_n')-\TT(t_n))}+\int_{t_n}^{t_n'}\neweta_\tau d\tau. 
\]
Using this estimate, the Lipschitz continuity of
$r_t$ established in Lemma \ref{lem-rtprops} and the induction hypothesis,
it follows that 
\begin{align*}
r_{t_{n+1}}
&\le r_{t_n'}+c_{\rm lip}(t_{n+1}-t_n')
\\
&\le
r_{t_n}e^{-\eps_h(\TT(t_n')-\TT(t_n))}+\int_{t_n}^{t'_n}\neweta_\tau d\tau+c_{\rm lip}(t_{n+1}-t'_n)
\\
&\le
\Big(
r_{t_1}e^{-\eps_h\sum_{i=1}^{n-1}(\TT(t_i')-\TT(t_i))}
+\sum_{i=1}^{n-1}\int_{t_i}^{t_i'}\neweta_\tau d\tau+c_{\rm lip}\sum_{i=1}^{n-1}(t_{i+1}-t'_i)
\Big)
e^{-\eps_h(\TT(t_n')-\TT(t_n))}
\\
&\qquad +\int_{t_n}^{t_n'}\neweta_\tau d\tau+c_{\rm lip}(t_{n+1}-t'_n)
\\
&\le
r_{t_1}e^{-\eps_h\sum_{i=1}^{n}(\TT(t_i')-\TT(t_i))}
+\sum_{i=1}^{n}\int_{t_i}^{t_i'}\neweta_\tau d\tau+c_{\rm lip}\sum_{i=1}^{n}(t_{i+1}-t'_i).
\end{align*}
This proves \eqref{09} by induction.

Next, note that each of the intervals $(t_0,t_1)$, $(t'_i,t_{i+1})$ for $i=1,\ldots, \numb_u-1$,
and $(t_{\numb_u},t)$, is a subset of $\calT_u$. Hence, we have 
\[
\TT(t)-\TT(t_0)-\sum_{i=1}^{\numb_u-1}(\TT(t'_i)-\TT(t_i))
\le (\TT(t_1)-\TT(t_0))+\sum_{i=1}^{\numb_u-1} (\TT(t_{i+1})-T(t_i')) +(\TT(t)-\TT(t_{\numb_u}))
\le |\calT_u|.
\]
Hence, on applying \eqref{09} with $n=\numb_u$,
noting that the last term on the right-hand side is bounded by
$c_{\rm lip}|\calT_u|$, we obtain on taking $u\to0$,
$r_t\le c_r e^{-\eps_h(\TT(t)-\TT(t_0))}+\int_{t_0}^t\neweta_\tau d\tau$.
Finally, since by definition $s\le t_0$ and $\TT(t_0)=\TT(s)$, the lemma follows.
\end{proof}

\paragraph{{\bf Step 3.}}
We now prove that $\bar{\TT} := \sup_t \TT(t) = \infty$. 
Note that this  implies that 
the ``servers'' become busy infinitely often, as one might expect in the supercritical regime $\lambda > 1$.
(We will later use this to prove  the stonger condition that the complement of  ${\mathcal B}$ is bounded.)

Arguing by contradiction, assume that $\bar{\TT} < \infty$.
By \eqref{01}, clearly $\lan \f1,\mu_t\ran\le\lan \f1,\nu_t\ran\le1$ and, by Assumption \ref{ass:a}(\ref{ass:a1}),   $\langle h, \nu_t \rangle \leq c_h$, for all $t > 0$.
However, \eqref{01}, \eqref{14-} and \eqref{fkprime} together implly 
\begin{align*}
\lan \f1,\mu_t\ran=\int_0^\infty \ttpsi(x,t)dx &=\int_0^t\bar G^s(x)k_{t-x}dx\\
&=\int_0^t\bar G^s(x)[\la 1_{\{B_{t-x}<1\}}+\lan h^s,\nu_{t-x}\ran
1_{\{B_{t-x}=1\}}]dx\\
&\ge
\la\int_0^t\bar G^s(x)dx -\la\int_0^t\bar G^s(x)1_{\{B_{t-x}=1\}}dx.
\end{align*}
Moreover, it is also true that 
\begin{align*}
\int_0^t\bar G^s(x)1_{\{B_{t-x}=1\}}dx
&\le \int_0^{t/2}1_{\{B_{t-x}=1\}}dx
+\int_{t/2}^t\bar G^s(x)dx\\
&\le (\TT(t)-\TT(t/2))+\int_{t/2}^\iy\bar G^s(x)dx.
\end{align*}
Recalling that $\int_0^\iy\bar G^s(x)dx=1$  (see Assumption \ref{ass-main}), if $\bar{\TT} < \infty$ 
the above expression converges to zero as $t \rightarrow \infty$. 
Hence, $\liminf_{t\to\iy}\lan \f1,\mu_t\ran\ge\la>1$, which is a contradiction. This proves $\bar{\TT} = \infty$.

\paragraph
    {\bf Step 4.} We now combine the above results to prove Proposition \ref{prop-superpf}.

    \begin{proof}[Proof of Proposition \ref{prop-superpf}.]
     We first claim that to establish \eqref{14}, it suffices to show that 
$B(t)=1$ for all sufficiently large $t$. 
   Recalling that  $\neweta = \moc (\langle h^s, \newdel \rangle)$ is integrable on $[0,\iy)$ by Lemma \ref{lem2}
and the bound \eqref{120} on   $\moc$, and that 
 $0\leq L(t)\to\iy$ as $t\to\iy$ by Step 3, which implies $\calB$ is unbounded,  we can 
 send first $t\to\iy$ along $\calB$  and then $s\to\iy$ in \eqref{ineq-lem3} of Corollary \ref{cor-lem3}, to obtain 
$\limsup_{t\to\iy,\ t\in\calB}r_t\le0$.
We cannot directly deduce from this that the limit of $r_t$ along
$\calB$ is zero, since
$r_t = R(\mu_t\| \nu_*)$ could be negative. 
However, for $t \in {\mathcal B}$,  $B(t) = \lan \f1,\nu_t\ran=1$ and 
hence, $\lan \f1,\mu_t\ran = 1- \lan \f1,\newdel_t\ran$.  Since $r_t = R(\mu_t\|\nu_*)$, and $\langle \f1, \newdel_t \rangle \to 0$ by
Remark \ref{rem-theta}, when combined with  \eqref{10}  this implies 
$\limsup_{t \to \iy, t \in \calB} r_t \geq \limsup_{t \to \iy, t \in \calB} \langle \f1, \mu_t\rangle \ln \langle \f1, \mu_t \rangle = 0$.
Hence, 
\begin{equation}
  \label{14-lite}
\lim_{t\to\iy,\ t\in\calB}\lan \f1,\mu_t \ran=1\quad \text{and} \quad
\lim_{t\to\iy,\ t\in\calB}r_t=0, 
\end{equation}
If $\calB \supseteq  [t_0,\infty)$ for some  finite $t_0$, this clearly proves \eqref{14}, and the claim follows.

We now turn to the proof of the fact that $B(t) = 1$ outside a finite interval. 
First note that,  \eqref{14-lite} and the Pinsker-type inequality \eqref{ineq-Pinsker} together show that 
\begin{equation*}
\lim_{t\to\iy,\ t\in\calB}\int_0^\iy|\ttpsi(x,t)-f^*(x)|dx=0.
 \end{equation*}
Thus, given $\eps_0 :=\frac{\lambda-1}{4}$ there exists  $T \in \calB$ such that
\begin{equation}\label{2:0}
c_h\lan \f1, \newdel_t\ran<\eps_0
\quad \text{ and } \quad
c_h\int_0^\infty|\ttpsi(x,t)-f^*(x)|dx<\eps_0
\qquad \text{for all } t\geq T,\ t\in\calB.
\end{equation}
We claim that $[T,\iy)\subset\calB$. Arguing by contradiction,
  assume there exists $T'> T$ for which $T' \not \in \calB$, that is, such that 
  $B ({T'})<1$. 
Let $\tau :=\sup \{t<T':B_t=1\}$. By the continuity of $B$, $T \leq \tau< T'$ and 
  $\tau\in\calB$; in particullar,  the estimates in \eqref{2:0} are valid for $t=\tau$.
Find $t^* > 0$ so small that $G^s(t^*)<\frac{1}{4}$ and $0<t^*< T'-\tau$. For all $0\le t\leq t^*$,
 applying Lemma \ref{lem:shift} with $T = \tau$ and  \eqref{f5} with $\psi = h^s$,  and using the fact that  \eqref{fkprime} implies
  $K (\tau+t) - K(\tau) =  \lambda t$ for $t \in (0,t^*)$, the identity $g^s = h^s \bar{G}^s$, the upper bound on $h^s$ from
  Assumption \ref{ass:a}(\ref{ass:a1}) and \eqref{2:0},  we obtain 
\begin{align*}
\lan h^s,\nu_{\tau+t}\ran&=\int_0^\infty\frac{g^s(x+t)}{\bar G^s(x)}\nu_\tau(dx)+\lambda\int_0^tg^s(t-s)ds\\
&=\int_0^\infty\frac{g^s(x+t)}{\bar G^s(x)}\newdel_\tau(dx)
+\int_0^\infty\frac{g^s(x+t)}{\bar G^s(x)}\nu^*(dx)
+\int_0^\infty\frac{g^s(x+t)}{\bar G(x)}(\mu_\tau(dx)-\nu^*(dx))
+\lambda\int_0^tg^s(t-s)ds
\\
&\leq c_h\lan \f1,\newdel_\tau\ran+\int_0^\infty g^s(x+t)dx+c_h\int_0^\infty|f(\tau,x)-f^*(x)|dx+\lambda G^s(t)\\
&\leq \eps_0+1-G^s(t)+\eps_0+\lambda G^s(t)
\\
&= 1+(\lambda-1)G^s(t)+2\eps_0
\leq 1+3\eps_0
=\lambda-\eps_0.
\end{align*}
Thus for all $\tau<s<\tau+t^*$, $\lan h,\nu_s\ran\le\lambda-\eps_0$.
Next, since the interval $(\tau, \tau+t^*)$ is a subset of an excursion,  
equation \eqref{1:1} for $B$ is valid for $t$ in that interval, and  it follows that 
$B'(s)\ge\eps_0$ for $s\in(\tau,\tau+t^*)$. 
Bsy the continuity of $B$,
\[
B (s) \ge 1+(t-\tau)\eps_0>1,\qquad s\in(\tau,\tau+t^*),
\]
which is a contradiction. We have thus shown that $B(t)=1$ for all sufficiently large $t$.  Together
with \eqref{14-lite}, this proves  \eqref{14}.

 To conclude the proof of the proposition, it only remains to show that $\lan h^s,\nu_t\ran\to1$ as $t \to \infty$.
Fix $T \in (0,\infty)$ such that $B(t)=1$ for all $t\ge T$.
Then using Lemma \ref{lem:shift}  and  equation \eqref{f5} with $\fsp(x)=h^s(x)$, and noting from \eqref{fkprime} that  
$K^\prime(T +s) = \langle h, \nu_{T+s}\rangle$ for all $s > 0$,  and recalling again that $g^s  = \bar{G}^s h^s$, we have 
\[
\lan h^s,\nu_{T+s}\ran = z(s)+\int_{[0,s]}g^s(T + s-w)\lan h^s,\nu_{T+w}\ran dw, 
\]
where $z(s) :=\int_{[0,\iy)}\frac{g^s(x+s)}{\bar G^s(x)}\nu_{T}(dx).$ 
Next, note that $g^s(s)=h^s(s)\bar G^s(s)\le c_h\bar G^s(s)$. Since $\bar G^s$
is decreasing and integrable over $[0,\iy)$, it is also directly Riemann integrable
(see Prop.\ 2.16(c), Ch.\ 9 of \cite{cin}), and thus, so is $g^s$.
Hence, by the key renewal theorem (Theor.\ 2.8, Ch.\ 9 of \cite{cin}),
$\lan h^s,\nu_{T+s}\ran$ converges as $s\to\iy$ to $\int_0^\iy z(s)ds/\int_0^\iy xg^s(x)dx=\int_0^\iy z(s)ds$,
since by Assumption \ref{ass-main}, $\int_0^\iy xg^s(x)dx=1$. Thus,
\begin{align*}
\lim_{s\to\infty}\lan h^s,\nu_{T+s}\ran &=\int_0^\iy\int_{[0,\iy)}\frac{g^s(x+s)}{\bar G^s(x)}\nu_{T}(dx)ds
\\&=\int_{[0,\iy)}\frac{1}{\bar G^s(x)}\Big(\int_0^\iy g^s(x+s)ds\Big)\nu_{T}(dx)
=\int_{[0,\iy)}\nu_{T}(dx), 
\end{align*}
which is equal to $1$ by our choice of $T$.  This completes the proof of the proposition. 
\end{proof}

    \subsection{Proof of Convergence when the Hazard Rate Function is Decreasing}
    \label{subs-pf2supcrit}

  In this section, we assume throughout that  Assumption \ref{ass-main} and 
  Assumption \ref{ass:a}(\ref{ass:a2})  hold, and we establish
  Theorem \ref{th-main2}(2) in this case, as well as Theorem \ref{th-main2}(3). 
    In addition, fix  $\lambda \geq  1$, and suppose that $(\fx,\fmeas,\freneg)$ is the  solution to the fluid equations with arrival rate $\lambda$ and some initial condition  $(\fx(0), \fmeas_0, \freneg_0) \in \newcspace$.  Also, recall from (\ref{def-q}) that 
$B(t) = \langle \f1, \fmeas_t\rangle$, and define 
	\begin{equation}
		\label{def-W}
		W(t) := B(t) - \int_{[0,H^s)} \frac{\bar{G}^s(x+t)}{\bar{G}^s(x)} \nu_0(dx), \quad t \geq 0. 
	\end{equation}
        Note that $W(t)$ represents the fluid mass of jobs that arrived after time $0$ and
are still in service at time $t$.

We will first establish the following key result.

\begin{proposition}
  \label{prop-key}
 Suppose  Assumption \ref{ass-main} and Assumption \ref{ass:a}(\ref{ass:a2})  hold, and $\lambda \geq 1$.
Then we have 
	\begin{equation}
		\label{Blim}
		\lim_{t \rightarrow \infty} W(t)= \lim_{t \rightarrow \infty} B(t) = 1, 
	\end{equation}
Further, if $\lambda > 1$, there exists $T \in [0,\infty)$ such that $B(t)=1$ for all $t\geq T$.
\end{proposition}

Before launching into the proof, we derive some useful relations.  
Setting $\fsp \equiv \f1$ in (\ref{f5}) and using
integration by parts, it follows that 
\begin{eqnarray}  
	B(t)  =  \langle \f1, \nu_t \rangle & = &
	\int_{[0,H^s)} \frac{\bar{G}^s(x+t)}{\bar{G}^s(x)} \nu_0(dx)
	+ \int_0^t \bar{G}^s(t-s) dK(s) \label{eq-bk} \\
	& = & \int_{[0,H^s)} \frac{\bar{G}^s(x+t)}{\bar{G}^s (x)} \nu_0(dx) + K(t) -
	\int_0^t K(s) g^s(t-s)ds,\nonumber 
\end{eqnarray}
which when rearranged yields 
\begin{equation}
	\label{rel-K}
	K(t) =W(t)
	+ \int_0^t K(t-s) g^s(s) ds. 
\end{equation}
Then, (\ref{eq-bk}),  \eqref{def-W} and the fact that $\fmeas_t$ is a sub-probability measure,
together imply that  for each $t\geq 0$,
\begin{equation}
	\label{rel-W}
	W(t)=\int_0^t \bar{G}^s(t-s) dK(s)\geq 0 \quad \mbox{ and } \quad
	W(t)\leq B(t) \leq 1.  
\end{equation}
Together with \eqref{rel-K} and  the renewal theorem (see Chapter V of \cite{ASS}), this implies 
\begin{equation}  \label{eq-kb} K(t) = W(t) + Z(t), \qquad
	\mbox{ with }  \quad 
	Z(t) := \int_{0}^t W(t-s) \ dU_{s}(s),  
\end{equation}
and $U_s$ is equal to the renewal function of the distribution  with density $g^s$.
Now,
\eqref{eq-fk} implies that
\[ D(t) := \int_0^t \langle h, \nu_s \rangle ds = B(0) - B(t) + K(t), \qquad t \geq 0. \]
Then by (\ref{eq-kb}), \eqref{def-W} and \eqref{eq-bk},
we obtain 
\begin{eqnarray}
	D(t) & = &  \langle \f1, \nu_0 \rangle - \int_{[0,H^s)}  \frac{\bar{G}^s(x+t)}{\bar{G}^s(x)} \nu_0(dx) + Z(t) \nonumber \\ 
	& = & \int_{[0,H^s)} \frac{G^s(x+t) - G^s(x)}{\bar{G}^s(x)} \nu_0(dx)
	+ Z(t).  \label{eq-d}
\end{eqnarray}
Under Assumption \ref{ass:a}(\ref{ass:a2}),  the hazard rate function $h^s$ is decreasing and hence,
by Theorem 3 of  \cite{Brown}, the renewal function $U_s$ is concave.  Since $G^s$ has density $g^s$,  the density $u_s:= U_s^\prime$ exists by Proposition 2.7 of \cite{ASS} and $u_s(x)= \sum_{n=1}^\infty (g^s)^{\star n}(x),\ x\geq 0$,
which in particular implies 
that $u_s(0) = g^s(0)$.  Moreover, by Alexandrov's Theorem (cf. page 172 of \cite{ncpl}), the concavity of $U_s$ implies  that $u_s$ is non-increasing, that is, 
\begin{equation}
	\label{rel-usprime}
	u_s^\prime (t) \leq 0, \quad \mbox{ for  a.e.\ } t \geq 0. 
\end{equation}
Now, differentiation of both sides of the defining equation for $Z$ in (\ref{eq-kb}) yields 
\begin{equation} Z^\prime (t) =   W(t)  u_s (0) + \int_0^t  W(t-s)  u_s^\prime (s) ds, \quad
	\mbox{ for  a.e.\ } t \geq 0.  \label{eq-z'}
\end{equation}
On the other hand, differentiating the  equation for $K$  in \eqref{eq-kb} and using \eqref{fkprime}, 
one obtains, for a.e.\ $t\geq 0$, 
\begin{eqnarray}
	W^\prime(t) & = &  K^\prime (t) - Z^\prime (t) \label{eq-bkz}\\
	& = & \left\{
	\begin{array}{ll}
		\lambda - Z^\prime (t)  & \mbox{ if } B(t) < 1, \\
		D^\prime (t) - Z^\prime (t) &  \mbox{ if } B(t) = 1  \mbox{ and } Q(t) > 0, \\
		\lambda \wedge D^\prime (t) - Z^\prime(t) & \mbox{ if } B(t) = 1 \mbox{ and } Q(t) = 0. 
	\end{array} 
	\right. \nonumber
\end{eqnarray}
Next, differentiating both sides of (\ref{eq-d}), we obtain for a.e. $t\geq 0$, 
\[ D^\prime (t) = \int_{[0,H^s)} \frac{g^s(x+t)}{\bar{G}^s (x)} \nu_0(dx) + Z^\prime (t)
\geq Z^\prime (t).
\]
Therefore,  by (\ref{eq-bkz}), for a.e. $t\geq 0$,
\begin{equation}
	\label{eq-conclusion}
	Z^\prime (t) \leq \lambda \quad \Rightarrow \quad  W^\prime(t) \geq 0.
\end{equation}

We  now establish some auxiliary results that will be used in the proof of Proposition \ref{prop-key}.

\begin{lemma} \label{lem:Wsmall}
  Suppose $\lambda\geq 1$. Then there is no $T \in (0,\infty)$ and $c\in (0,1)$ such that $W(t)<c$ for all $t\geq T$. The same assertion also holds when $W$
  is replaced with $B$.
  \end{lemma}
\begin{proof} Suppose the statement of the lemma is not true, that is, suppose there exists $T> 0$ and $c\in (0,1)$ such that $W(t)<c$ for all $t\geq T$. Since  $\int_{[0,H^s)} \frac{\bar{G}^s(x+t)}{\bar{G}^s(x)} \nu_0(dx) \rightarrow 0$ as $t\rightarrow \infty$, by \eqref{def-W},  there exists $T'>T$ such that $B(t)<1$ for all $t\geq T'$.
	In turn, by \eqref{eq-fk}, it follows that  
	$K'(t)=\lambda$, for all $t\geq T'$, and hence (\ref{eq-bk}) and (\ref{def-W}) imply
	that
	\[W(t)=\int_0^t \bar{G}^s(t-s) dK(s) = \int_0^{T'}  \bar{G}^s(t-s) dK(s) + \lambda \int_{T'}^t  \bar{G}^s(t-s) ds.
        \]
 As $t \rightarrow \infty$, the first term converges to zero by the  dominated convergence theorem and the pointwise limit  $\bar{G}^s(t-s) \rightarrow 0$. For the same reason, the second term converges to
         $\lim_{t \to \infty} \lambda \int_{0}^t  \bar{G}^s(t-s) ds = \lambda \int_0^\infty \bar{G}^s (s) ds$,
        which is equal to  $\lambda$ by \eqref{def-mean1} of Assumption \ref{ass-main}. 
    Thus, $\lim_{t \rightarrow \infty} W(t) \geq \lambda$, 
        which is a contradiction, thus proving the first assertion of the lemma.  Since, by \eqref{def-W}, $B(t) - W(t) = \int_{[0,H^s)} \frac{\bar{G}^s(x+t)}{\bar{G}^s(x)} \nu_0(dx) \rightarrow 0$ as $t\rightarrow \infty$, the same
          assertion holds also for $B$. 
\end{proof}

Next, substituting into  (\ref{eq-z'})  the inequality  \eqref{rel-usprime},    the relation
$u_s(0) = g^s(0)$ and the fact that $W(t)\in [0,1]$ for each $t\geq 0$ 
due to  \eqref{rel-W},  we see that 
\begin{equation}
	Z^\prime (t) \leq  W(t) u_s (0) = W(t) g^s(0) \leq g^s(0) \quad \mbox{for a.e.\ } t\geq 0.  \label{z'est}
\end{equation}
We also observe that since the hazard rate function $h^s$ is decreasing by Assumption \ref{ass:a}(\ref{ass:a2}), then $g^s(0)>0$. (Otherwise, if $g^s(0)=0$, then $h^s(0)=0$, which implies that $0\leq h^s(t)\leq h^s(0)=0$ for each $t\geq 0$ and thus, $g^s(t)=0$ for all $t\geq 0$, which would contradict the fact that $g^s$ is the density of $G^s$.)  Therefore,  for $n \in \N \cup \{0\}$ and $\varepsilon \in (0,\frac{1}{2})$, define
\begin{equation}  \label{lambda-n} \lambda_n := \frac{\lambda-\varepsilon}{g^s(0)} \left( \sum_{i=0}^n 
	\left(  1 - \frac{1}{g^s(0)} \right)^i \right)= \frac{\lambda-\varepsilon}{g^s(0)} \frac{1 - \left(  1 - \frac{1}{g^s(0)} \right)^{n + 1}}{1  - \left(  1 - \frac{1}{g^s(0)} \right)} = (\lambda -\varepsilon) \left(1 - \left(  1 - \frac{1}{g^s(0)} \right)^{n + 1} \right), \end{equation}
and
\begin{equation}
	\label{tau-n}
	\tau_n := \sup \{ t > 0: W(t) < \lambda_n \}. 
\end{equation}
If $\tau_n < \infty$, then 
\begin{equation}
	\label{wtaun}
	W(\tau_n + t)  \geq \lambda_n \quad \forall t \geq 0. 
\end{equation}

\begin{lemma} \label{lem:preprop}
  Suppose $\lambda \geq 1$, $\varepsilon \in (0,\frac{1}{2})$ and $g^s(0)>\lambda-\varepsilon$. Then  $\tau_n<\infty$ and hence, \eqref{wtaun} holds for all $n \in \N$ with $n < n^*$, where
$n^* :=  \sup \left \{n \in \N_0: \lambda_n <  1 \right\}$,
and also for $n = n^*$, if $n^* < \infty$.
\end{lemma}
\begin{proof}
	Since $g^s(0)>\lambda -\varepsilon > \frac{1}{2}$ by the assumptions of the lemma, 
	it follows  that  $| 1 - \frac{1}{g^s(0)} |<1$ and  (\ref{lambda-n}) then implies that 
	\begin{equation} \lambda_n  \rightarrow \lambda-\varepsilon \mbox{ as } n\rightarrow \infty. \label{lim-lambdan} \end{equation}
	We prove the lemma by induction.  
	We first start with the base case $n=0$, where $\lambda_0= (\lambda-\varepsilon)/g^s(0)$.   
	Note that $\lambda_0 < 1$ by the assumptions of the lemma. We argue by contradiction to show that 
	\begin{equation}
		\label{wtau}
		W(t) < \frac{\lambda-\varepsilon}{g^s(0)} \quad \mbox{ for all } t \in (0,\tau_0). 
	\end{equation}
	Note that \eqref{wtau} holds trivially if $\tau_0 = 0$.  So, suppose $\tau_0 > 0$ and
	\eqref{wtau} does not hold. 
	Then there must exist  $0 < t_1 < \tau_0$ for which 
	$W(t_1) \geq \frac{\lambda-\varepsilon}{g^s(0)}$. It follow from \eqref{z'est} and \eqref{eq-conclusion}, that for a.e.\ $t\in (t_1,\tau_0)$,  the inequality $W(t) < \frac{\lambda-\varepsilon}{g^s(0)}$ 
	implies that 
	$Z^\prime(t) \leq W(t) g^s(0) < \lambda -\varepsilon <\lambda$  and hence
	$W^\prime (t) \geq  0$. Since $W$  is absolutely continuous, $W(t) \geq \frac{\lambda-\varepsilon}{g^s(0)} = \lambda_0$ for all
         $t \in [t_1,\tau_0)$ (see Lemma \ref{lem:f}).  This contradicts 
	the definition of $\tau_0$, and thus, \eqref{wtau} holds.   If  $\tau_0 = \infty$, then   \eqref{wtau} implies $W(t)<\lambda_0<1$ for all $t>0$,
	which contradicts Lemma \ref{lem:Wsmall}.  Thus, $\tau_0<\infty$. This completes the proof of the base case. 
	
	Now, suppose that  $\tau_k<\infty$ for some $k \in \N \cup \{0\},$ with $k < n^*$ if $n^* < \infty$. It follows that $\lambda_{k+1}<1$ by the choice of $k$ and the  definition of $n^*$. 
	By the definition of $\tau_k$ and the continuity of $W$, 
	\begin{equation}
		\label{claimW}
		W(\tau_k + t) \geq \lambda_k\mbox{  for all } t \in [0,\infty).
	\end{equation}
	Then for a.e.\ $t \geq 0$, by \eqref{eq-z'}, \eqref{rel-usprime},  \eqref{claimW} 
	and the relations $W(t) \geq 0$ and $u_s(0) = g^s(0)$, we have  
	\begin{eqnarray}
		Z^\prime (\tau_k + t) & = & W(\tau_k +t) g^s(0) + \int_0^t W(\tau_k+t-s) u_s^\prime (s) ds+ \int_t^{\tau_k + t} W(\tau_k + t -s) u_s^\prime(s)ds  \label{eqn-Z'est1}\\
		& < & W(\tau_k+t) g^s(0) + \lambda_k
		\left( u_s(t) - g^s(0)\right).  \label{eqn-Z'est}
	\end{eqnarray}
Since Assumption \ref{ass:a}(\ref{ass:a2}) implies that the integrable function $g^s$ is also bounded, it lies in $\mathbb{L}^{1+\varepsilon}(0,\infty)$
for any $\varepsilon > 0$, and  satisfies $g^s(t) \rightarrow 0$ as $t\rightarrow \infty$. Thus,  by Theorem 12 of \cite{smith54} we
can conclude that $\lim_{t \rightarrow \infty} u_s(t) = 1$.  
	Hence,  there exists $\sigma_k>0$ such that
	\begin{equation}  
		\label{eq-sigmak}
		(\lambda -\varepsilon) + \lambda_k (u_s(t) -1) = (\lambda -\varepsilon) +
		\lambda_k\left(  g^s(0) - 1 \right) + \lambda_k
		\left( u_s(t) - g^s(0)\right) <\lambda \mbox{ for all } t\geq \sigma_k.
	\end{equation}  
	We now show that the following statement cannot hold: 
	\begin{equation}
		\label{liminfW1}
		W(\tau_k+t)  <   \lambda_{k+1} = 
		\frac{\lambda-\varepsilon}{g^s(0)}  +
		\lambda_k\left(  1 - \frac{1}{g^s(0)} \right) \mbox{ for all } t> \sigma_k,  
	\end{equation}
 where the  equality follows from \eqref{lambda-n}. 
	Indeed, if this were true, then   this  would imply that  
	$W(t)<\lambda_{k+1}<1$ for all $t \geq \sigma_k':=\tau_k+\sigma_k$, which contradicts 
	Lemma \ref{lem:Wsmall}.  Thus,   \eqref{liminfW1} does not hold or, in other words,
	there exists $\tau_k' \in (\sigma_k',\infty)$ such that
	$W(\tau_k') \geq \lambda_{k+1}$. 
     We now show that for a.e.\ $t \in (0, \infty)$, if $W(\tau_k' + t) < \lambda_{k+1}$ then $W^\prime (\tau_k' + t) \geq 0$.
        Indeed, if the first inequality is true, then substituting
	this into  (\ref{eqn-Z'est}) with $\tau_k'$ in place of $\tau_k$, and using
	\eqref{eq-sigmak}, it follows that $Z^\prime (\tau'_k+t)<\lambda$.
        When combined with  \eqref{eq-conclusion} the latter implies $W^\prime (\tau_k' + t) \geq 0$.
        Hence (applying Lemma \ref{lem:f} with $f=W$, $c= \lambda_{k+1}$, $T= \tau_k'$, $S=\infty$), 
        it follows that $W(\tau_k'+t) \geq \lambda_{k+1}$ for all $t \geq 0$, thus showing that 
         $\tau_{k+1}\leq \tau_k'<\infty$.
	By induction, it follows that  for each $0\leq n <   n^*$, $\tau_n<\infty$ and hence, \eqref{wtaun} holds,
	and if $n^* < \infty$ then also $\tau_{n^*} < \infty$ and \eqref{wtaun} holds with
	$n = n^*$. This completes the proof of the lemma.
	
\end{proof}

We are now in a position to present the proof of Proposition \ref{prop-key}. 

\begin{proof}[Proof of Proposition \ref{prop-key}] 
	We  first prove the proposition when $\lambda=1$. For this, we consider two cases. \\
	{\em Case 1a: }  $g^s(0) \leq 1$.
	In this case, \eqref{z'est} shows that $Z^\prime (t) \leq  1 $ for a.e.\  $t\geq 0$,  then \eqref{eq-conclusion} implies that  for
	a.e.\ $t \geq 0$, $W^\prime(t) \geq 0$.    Since $W$ is absolutely continuous by \eqref{rel-W}
	this implies that  $W$ is increasing on $[0,\infty)$ and  $b := \lim_{t \rightarrow \infty} W(t)$ exists. Furthermore, 
	\eqref{def-W} and the fact that $\bar{G}^s(x+t) \rightarrow 0$ as $t \rightarrow \infty$ for
	every $x \in [0,H^s)$, imply 
	$b = \lim_{t \rightarrow \infty} B(t).$ 
	We now argue by contradiction to show that $b = 1$.  Suppose    
	$b < 1$, then for any $T < \infty$, there exists  $T_1 < \infty$ such that for $t \geq 0$, 
	$B(T_1+t) < 1$ and  thus, by
        \eqref{fkprime} 
	$K^\prime (T_1 +t) = 1$.    Now, recalling $B(\cdot) = \langle \f1, \nu_{\cdot} \rangle$ from 
	\eqref{eq-bk} and combining  Lemma \ref{lem:shift} and Theorem \ref{th-fluid},
	it follows that \eqref{f5} holds with $\fsp = \f1$, and $\fmeas_t$ and $K_t$ 
	replaced with $\fmeas_{T_1 +t}$, and $K_{T_1+t} - K_{T_1}$, respectively, or in other words, 
	for each $t\geq 0$,   
	\begin{eqnarray*}
		B(T_1 + t) & = & \int_{[0,H^s)} \frac{\bar{G}^s (x+t)}{\bar{G}^s(x)} \nu_{T_1}(dx)
		+ \int_0^t \bar{G}^s(t-s) K^\prime (T_1+s) ds.
	\end{eqnarray*}
	When combined with the  relation $K^\prime (T_1 +\cdot) = \lambda = 1$ a.e., this implies
	that for each $t \geq 0$. 
	\begin{eqnarray*}
		B(T_1 + t) 	& = & \int_{[0,H^s)}  \frac{\bar{G}^s(x+t)}{\bar{G}^s(x)} \nu_{T_1} (dx) +
		\int_0^t \bar{G}^s(t-s) ds,
	\end{eqnarray*}
         Sending $t \rightarrow \infty$, using  $\bar{G^s}(x+t) \to 0$ pointwise 
        and the dominated convergence theorem, as well as \eqref{def-mean2} of Assumption \ref{ass-main},  
this implies	$b = \lim_{t \rightarrow \infty} B(t) = 1$.
	This contradicts the supposition that $b< 1$, and thus proves 
	that $b=1$.

        \noindent 
	{\em Case 1b: }  $g^s(0) > 1$. In this case, by  (\ref{lambda-n}),
	\[\lambda_n = (1 -\varepsilon) \left(1 - \left(  1 - \frac{1}{g^s(0)} \right)^{n + 1} \right) <1 \mbox{ for all } n\geq 1. \] Thus, by Lemma \ref{lem:preprop}, for each $n\geq 1$, we have $\tau_n<\infty$ and so \eqref{wtaun} implies
	$\liminf_{t\rightarrow \infty} W(t) \geq \lambda_n \mbox{ for each } n\geq 1$.
	By (\ref{lim-lambdan}), we obtain 
	$\liminf_{t\rightarrow \infty} W(t) \geq 1-\varepsilon$.  Sending $\varepsilon\downarrow 0$, we obtain $\liminf_{t\rightarrow \infty} W(t) \geq 1$.
	Since $\limsup_{t\rightarrow \infty} W(t) \leq 1$ by \eqref{rel-W} it follows that
	in fact   $\lim_{t\rightarrow \infty} W(t)=1$. When combined with
	\eqref{def-W} and the fact  that $\bar{G}^s(x+t) \rightarrow 0$ as $t \rightarrow \infty$ for
	every $x \in [0,H^s)$, it follows that 
	$\lim_{t \rightarrow \infty} B(t) = 1,$ thus proving the proposition in this case.\\
	
	We next prove the proposition for the case that  $\lambda>1$.  Let $\varepsilon > 0$ be small enough such that $\lambda - \varepsilon>1$.    We now consider two cases. \\
	{\em Case 2a: }    $g^s(0) \leq \lambda -\varepsilon$.
	In this case, \eqref{z'est} shows that $Z^\prime (t) \leq \lambda-\varepsilon < \lambda $ for a.e.\  $t\geq 0$,  and hence, \eqref{eq-conclusion} implies that  for
	a.e.\ $t \geq 0$, $W^\prime(t) \geq 0$.  Moreover,  by (\ref{eq-bkz}), we have
	$W^\prime ( t) = \lambda -Z^\prime ( t) \geq \varepsilon \mbox{ if } B(t)<1$.
	By the definition of $W$ in (\ref{def-W}), we obtain \[B^\prime  ( t)  =  W^\prime  (t) + \int_{[0,H^s)} \frac{g^s(x+ t)}{\bar{G}^s(x)} \nu_0(dx). \]
	Since $h^s$ is decreasing, we have  $h^s(x+ t) \leq h^s(0)$ for each $x\in [0, H^s-t)$, and an application  of the dominated convergence theorem shows that \[ \int_{[0,H^s)} \frac{g^s(x+ t)}{\bar{G}^s(x)} \nu_0(dx) \leq \int_{[0,H^s)} \frac{h^s(0)\bar{G}^s(x+ t)}{\bar{G}^s(x)} \nu_0(dx) \rightarrow 0 \mbox{ as } t\rightarrow \infty. \]
	The last three displays together imply that
	there exists $T \in (0,\infty)$ 
	such that $B^\prime (t) > \varepsilon/2$ whenever $B(t) <1$ for a.e.\  $t \in [T,\infty)$.  Since $B$ is bounded (by $1$), the inequality 
          $B(t)<1$ cannot hold for all $t\geq T$.  In other words, 
            there must exist $T^\prime >T$ such that $B(T^\prime)=1$.
   Since $B$ is absolutely
           continuous and bounded by $1$ (applying Lemma \ref{lem:f} with $f=B$, $c=1$, $T=T'$ and $S=\infty$),
         we conclude that $B(t)=1$ for all $t\in [T^\prime,\infty)$.

	\noindent 
	{\em Case 2b: } $g^s(0) > \lambda -\varepsilon$. 
	Then $n^* < \infty$ since \eqref{lambda-n} shows that
	$\lambda_n \uparrow (\lambda - \varepsilon) > 1$ as $n \rightarrow \infty$. Since by Lemma \ref{lem:preprop}, $\tau_{n^*}<\infty$,  then 
	 the continuity of $W$ dictates that  $W(\tau_{n^*})=\lambda_{n^*}$.
	Together with (\ref{eqn-Z'est}) with $k = n^*$ 
	and the fact that $W$ is bounded by $1 $ due to  \eqref{rel-W}, this implies that for
	a.e.\ $t \geq 0$, 
	\begin{eqnarray*}
		Z^\prime (\tau_{n^*} + t) & \leq  & W(\tau_{n^*}+t) g^s(0) + \lambda_{n^*}
		\left( u_s(t) - g^s(0)\right) \\ &\leq & 
                (1-\lambda_{n^*}) g^s(0) + \lambda_{n^*}  u_s(t).
	\end{eqnarray*}
	By the definition of $n^*$, we have $\lambda_{n^*}<1 \leq \lambda_{n^*+1}$.
	Together with the definition of $\lambda_n$ in (\ref{lambda-n}), this implies that 
	\[1- \lambda_{n^*} \leq \lambda_{n^*+1} - \lambda_{n^*}  = \frac{\lambda-\varepsilon}{g^s(0)} 
	\left(  1 - \frac{1}{g^s(0)} \right)^{n^*+1}.\] Combining the above two displays, we obtain
	\[Z^\prime (\tau_{n^*} + t)  \leq (\lambda-\varepsilon) 
	\left(  1 - \frac{1}{g^s(0)} \right)^{n^*+1} + \lambda_{n^*}  u_s(t).\]
	Recalling that $\lim_{t \rightarrow \infty} u_s(t) = 1$ and using the expression for
	$\lambda_{n^*}$ from \eqref{lambda-n}, it follows that as $t \rightarrow  \infty$, 
	\[(\lambda-\varepsilon) 
	\left(  1 - \frac{1}{g^s(0)} \right)^{n^*+1} + \lambda_{n^*}  u_s(t) \rightarrow (\lambda-\varepsilon) 
	\left(  1 - \frac{1}{g^s(0)} \right)^{n^*+1} + \lambda_{n^*} = \lambda -\varepsilon.  \]
	Thus,   for all $t$ large enough, $Z^\prime (\tau_{n^*}+ t)<\lambda -\varepsilon/2$.  However, note that by (\ref{eq-bkz}), we have \[W^\prime (\tau_{n^*}+ t) = \lambda -Z^\prime (\tau_{n^*}+ t) >\varepsilon/2 \mbox{ if } B(\tau_{n^*}+ t)<1. \]
	By the definition of $W$ in (\ref{def-W}), we obtain \[B^\prime  (\tau_{n^*}+ t)  =  W^\prime  (\tau_{n^*}+ t) + \int_{[0,H^s)} \frac{g^s(x+\tau_{n^*}+ t)}{\bar{G}^s(x)} \nu_0(dx). \]
	  Since $h^s$ is decreasing, it follows that $h^s(x+\tau_{n^*}+ t) \leq h^s(0)$ for each $x\in [0,H^s)$, and we obtain by the dominated convergence theorem that \[ \int_{[0,H^s)} \frac{g^s(x+\tau_{n^*}+ t)}{\bar{G}^s(x)} \nu_0(dx) \leq \int_{[0,H^s)} \frac{h^s(0)\bar{G}^s(x+\tau_{n^*}+ t)}{\bar{G}^s(x)} \nu_0(dx) \rightarrow 0 \mbox{ as } t\rightarrow \infty. \]
                The rest of the proof follows as in Case 1b. 
	The last four displays imply that
	there exists $T \in (0,\infty)$
	such that $B^\prime (t) > \varepsilon/4$ whenever $B(t) <1$ for a.e.\  $t \in [T,\infty)$.
          By the boundedness of $B$ it follows that there exists $T^\prime >T$ such that $B(T^\prime) = 1$. Thus, for  a.e.\  $t\geq T^\prime$, we have $B^\prime (t) > \varepsilon/4$ whenever $B(t) <1$.  In turn (by Lemma \ref{lem:f}) this implies that $B(t)=1$ for all $t\in [T^\prime,\infty)$. 
	Since  all possible cases have been considered, this concludes the proof of  the proposition. 
\end{proof}

We now consider convergence properties of the  measure-valued age process. 

\begin{lemma}
	\label{cor-supercritlim}
           For $\lambda \geq 1$,  under the assumptions of Proposition \ref{prop-key}, suppose there exists
              $T < \infty$ such that $B(t)= 1$ for all $t \geq T$.   Then
              $\nu_t \Rightarrow \fmeass$ and $\lan  h^s, \fmeas_t \ran \rightarrow 1$ as $t\rightarrow \infty$. 
\end{lemma}
\begin{proof}
  By invoking Lemma \ref{lem:shift}, we can assume  without loss of generality that $T=0$.
  Then $B(t) = \langle \f1, \nu_t \rangle = 1$ for all $t \geq 0$,
and so by (4.6) of Corollary 4.4 of \cite{KasRam11},
$K$ has the representation
	\[
	K (t)      = \int_0^t \left(
	\int_{[0,H^s)} \frac{G^s(x+t-s) -G^s(x)}{\bar{G}^s(x)} \nu_0(dx) \right) dU_s(dx),\ t\geq 0.\]
          In view of the representation for the fluid age measure in \eqref{f5}, 
	  the convergence $\nu_t \Rightarrow \fmeass$ is then a direct consequence of Lemma 6.2 of  of \cite{KasRam11} with $\pi=\nu$.
    Finally, since $h^s$ is bounded and monotone by Assumption \ref{ass:a}(\ref{ass:a2}),  the set of its discontinuities
           is countable and thus has zero Lebesgue measure. Since $\nu_*$ is an absolutely continuous  measure, the continuous  mapping theorem
           implies $\langle h^s, \nu_t \rangle \rightarrow \langle h^s, \nu_*\rangle = \int_0^\infty g^s(x) dx = 1$, as $t \to \iy$. This  concludes 
         the proof of the lemma. 
\end{proof}

\subsection{Uniqueness of Random Fixed Points}
\label{subs-rfp}

We now show how  the convergence results of the last two sections
can be bootstrapped to conclude, under Assumption \ref{ass-unique},
the existence of a unique random fixed point.

\begin{proposition} 
	\label{cor-randomfp}
        Suppose  $\lambda\geq 1$, Assumptions \ref{ass-main} and  \ref{ass-unique} hold and 
        suppose that for any solution $(X,  \fmeas, \freneg)$  to the fluid equations
        with arrival rate $\lambda$ and initial condition   $(\fx(0), \fmeas_0, \freneg_0) \in \newcspace$, 
        \begin{equation}
          \label{rfp-ass}
          \freneg_t \Rightarrow \lambda \freneg_* \quad \mbox{ and } \quad B_t \rightarrow 1. 
        \end{equation}
         Then any random fixed point $\mu$ for the fluid equations with arrival rate $\lambda$ 
       satisfies $\mu = \delta_{z^\lambda_*}$, where $z^\lambda_* =  (x^\lambda_*, \fmeass, \lambda \frenegs)$, with $x^\lambda_*$ being the unique
        element of $\fplambda_\lambda$ in \eqref{eq-fp}.    
\end{proposition}  
\begin{proof}
  Fix $\lambda \geq 1$ and let $\mu$ be a random fixed point for the fluid equations with arrival rate $\lambda$.  
  Let $(\fx(0),\fmeas_0, \freneg_0)$ be a random element taking values in $\R_+ \times {\mathcal M}_F[0,H^s) \times {\mathcal M}_F[0,H^r)$ with
      law $\mu$ and let  $(\fx, \fmeas, \freneg)$ be the solution to the fluid equations with arrival rate $\lambda$ and initial condition   $(\fx(0), \fmeas_0, \freneg_0) \in \newcspace$. 
      Since $\eta_t \Rightarrow \lambda \frenegs$ and $B_t \rightarrow 1$ by assumption and the laws of 
      $\eta_t$ and $\fmeas_t$ are invariant in $t$ since $\mu$ is a random fixed point, we have 
      $\PP( \freneg_0 = \lambda \frenegs) = 1$ and  $\PP(B_t = \langle \f1, \fmeas_t \rangle = 1)  = 1$.  Further, by continuity of $B$,  we have
      $\PP$-almost surely,  $B_t= 1$ for  all $t \geq 0$. Then by Lemma \ref {cor-supercritlim}, it follows that  $\nu_t \Rightarrow \fmeass$ as $t\rightarrow \infty$. Since the law of  $\fmeas_t$ is invariant in $t$, it follows
      that $\PP (\fmeas_0 = \fmeass) = 1$.

      To complete the proof, it only remains to show that $\PP(X(0)=x_*^\lambda)=1$.
      Since almost surely, for all $t\geq 0$,  $B(t)=1$ and $\freneg_t=\lambda \frenegs$,
     the relations (\ref{fq}) and
      (\ref{fr}) show that   almost surely for all $t\geq 0$, $X(t)=Q(t)+1$   and 
        \[\fr(t) = \int_0^t \left(\int_0^{\fq(s)}h^r(\left(F^{\lambda \frenegs}
	\right)^{-1} (y))dy\right) ds = \lambda \int_0^t G^r(\left(F^{\lambda \frenegs}
	\right)^{-1} (Q(s)) )ds. \]
	Moreover, using the fact that almost surely for each $t\geq 0$,  $\fmeas_t = \fmeass$, and hence, $D(t) = t \langle h^s, \fmeass\rangle = t$,
        we have from  \eqref{fq}, (\ref{eq-fx}) and the fact that $E = E^\lambda$ that almost surely for each $t\geq 0$,  
	\begin{eqnarray} Q(t)&=& Q(0) +(\lambda-1)t - \lambda \int_0^t G^r(\left(F^{\lambda \frenegs}
		\right)^{-1} (Q(s)) )ds \\ &=& Q(0) +\int_0^t \left(\lambda {\bar G}^r(\left(F^{\lambda \frenegs}
		\right)^{-1} (Q(s)) ) -1 \right)ds. \label{eq-Qdev}\end{eqnarray}
	We now consider two cases. \\
	{\em Case 1: }  $\lambda=1$. In this case, we have 
	\[\int_0^t \left(\lambda {\bar G}^r(\left(F^{\lambda \frenegs}
	\right)^{-1} (Q(s)) ) -1 \right)ds = -\int_0^t G^r(\left(F^{ \frenegs}
	\right)^{-1} (Q(s)) ) ds.\]	It is clear from (\ref{eq-Qdev}) that $Q$ is decreasing on $[0,\infty)$. By the non-negativity of $Q$, $q_*:=\lim_{t\rightarrow \infty}Q(t)$ exists and the fact that $X(t) = Q(t) +1$ implies $\lim_{t \rightarrow \infty} X(t) = x_* := q_* + 1$. Note that $G^r(\left(F^{\frenegs}	\right)^{-1} (q_*) ) =0$ since, otherwise, $Q(t)\rightarrow -\infty$ as $t\rightarrow \infty$, which contradicts the non-negativity of $Q$. 
          Therefore, by  Assumption \ref{ass-unique}, the definition of $\fplambda_\lambda$ in \eqref{eq-fp} 
          and the fact  that  $\lambda - 1 = 0$, it follows that $x_*$ is equal to  the unique element $x_*^\lambda$ of $\fplambda_\lambda$.
As before, since $\mu$ is a random fixed point,  this   implies that $\PP(X(0)=x_*^\lambda)=1$.  \\ 
{\em Case 2: }  $\lambda>1$. In this case, it is clear from (\ref{eq-Qdev}) that $Q$ is differentiable on $(0,\infty)$, and
	\begin{equation}
		Q^\prime (t) = \lambda {\bar G}^r(\left(F^{\lambda \frenegs}
		\right)^{-1} (Q(s)) ) -1 \mbox{ for each } t>0.
	\end{equation}
 First note that since by \eqref{fqfreneg}, $Q(t) \leq \langle \f1, \freneg_t \rangle = \lambda \langle \f1, \freneg_* \rangle$,
              $Q(t)$ is bounded. 
            We now argue by contradiction to show that $q_*=\lim_{t\rightarrow \infty}Q(t)$ exists. Suppose this is not the case. Then, since $Q$ is
            bounded on $[0,\infty)$,  $Q$  must oscillate for the limit  not to exist.   By the continuity of $Q$,   this implies there  must 
          exist two sequences of times $\{t_n, n\geq 1\}$ and $\{s_n, n\geq 1\}$ with $t_n \rightarrow \infty$, $s_n\rightarrow \infty$ as $n\rightarrow \infty$, and  $\varepsilon > 0$ such that $|Q(t_n)-Q(s_n)|>\varepsilon$ for   all $n$ sufficiently large, 
          and $Q^\prime (t_n)=Q^\prime (s_n)=0$ for each $n\geq 1$.  By \eqref{eq-Qdev}, the latter relation implies 
          \[  G^r(\left(F^{\lambda \frenegs}
	\right)^{-1} (Q(t_n)) ) = G^r(\left(F^{\lambda \frenegs}
	\right)^{-1} (Q(s_n)) ) = \frac{\lambda - 1}{\lambda} \mbox{ for all } n\geq 1. \]
	Since $Q$ is bounded, there exist $0 \leq q_i \leq \lambda \langle \f1, \eta^r \rangle$, $i = 1, 2$ and 
        is a subsequence $\{n_k,k\geq 1\}$ 
        such that $Q(t_{n_k})\rightarrow q_1$ and $Q(s_{n_k})\rightarrow q_2$ as $k \rightarrow \infty$. 
        It follows that $|q_1-q_2|\geq \varepsilon$ and $ G^r(\left(F^{\lambda \frenegs}
	\right)^{-1} (q_1))  = G^r(\left(F^{\lambda \frenegs}
	\right)^{-1} (q_2 ))=\frac{\lambda-1}{\lambda}$, where
        we have used the fact that $G^r$  and $\left(F^{\lambda \frenegs} \right)^{-1}$  are continuous,
          with the latter continuity holding because  $\lambda \frenegs$ has a  density $\lambda \bar{G}^r$ that is strictly positive on its support. 
          By Assumption \ref{ass-unique}, we have $q_1=q_2$ which contradicts $|q_1-q_2|\geq \varepsilon$. Thus, $q_*=\lim_{t\rightarrow \infty}Q(t)$ exists and then $\lambda {\bar G}^r(\left(F^{\lambda \frenegs}\right)^{-1} (Q(t)) ) -1=0$ since otherwise by \eqref{eq-Qdev} $Q$ will not have a limit. We can then argue as in Case 1 that
          $X(t)\rightarrow q_* + 1 = x_*^\lambda$, and thus $\PP(X(0) = x_*^\lambda) = 1$. 
	This completes the proof of the theorem. 
\end{proof}

\section{Results Regarding the Multiclass Model}
\label{sec-multiclass}
\beginsec

Here we consider the model with multiple classes operating under a fixed priority discipline. This model,
with general class-dependent service time  and patience time distributions, 
was analyzed in \cite{AtaKasShi13} and convergence at the fluid scale, uniformly
on compact time intervals, was established.
Here, we study the long-time behavior under the additional assumption that 
the service time distribution does not depend on the job class,  and its hazard rate function is bounded away from zero and infinity, 
	that is, satisfies  Assumption \ref{ass:a}(\ref{ass:a1}).   
	For simplicity of exposition, we also assume that the reneging distributions are exponential (but may depend
	on the job
	class) since the main motivation is to deduce the optimality of a certain priority 
	scheduling rule (known as the $c\mu/\theta$ rule; see details below) discussed in \cite{AtaKasShi13}, which is
	not expected to hold beyond the exponential reneging case. As shown in \cite{AtaKasShi13},  this optimality result relies on the 
	convergence of the invariant distributions of  the fluid-scaled process, as $N \to \infty$,
	to the  unique element of the invariant manifold of the fluid limit (under assumptions that ensure
	such uniqueness).  However, the convergence result in  \cite{AtaKasShi13} (specifically Theorem 4.3 therein)
	suffers from the same flaw as that described for the single-class case in Remark \ref{rem-fix};
	namely, from the proof in \cite{AtaKasShi13} one can only deduce that the invariant distributions of the $N$-server systems exist, 
	are tight and that any subsequential limit of the sequence of invariant distributions must be a {\em random} fixed point of the fluid equations
	(defined analogously to Definition \ref{def-RFP}).  As explained in Remark \ref{rem-fp} in the single-class setting, 
	in order to show that there is a unique random fixed point
	(which must then coincide with the unique element of the invariant manifold)
	it suffices to establish the long-time convergence of the solution of the fluid equations with any initial condition
	to the unique element of the invariant manifold.  Thus, the limit interchange result that  we prove
	here fixes the flaw in  the main optimality result of \cite{AtaKasShi13} under the additional  assumptions
	on the service distribution stated above.  This leaves open the  question of whether there is also a
	limit interchange for class dependent service times and when
	hazard rates are not necessarily bounded.
	We present the fluid equations in Section \ref{subs-multi1}, and
	then state and prove the theorem in Section \ref{subs-multi2}.

\subsection{Fluid Model Equations for the Multiclass System}
\label{subs-multi1}

Analogous to the single-class case, 
for each class $i\in\{1,\ldots,J\}$, we
denote by $B_i$, $X_i$ and $Q_i$  nonnegative functions 
that represent the fluid analogs of the number in service, number in system and number in queue, let the 
nonnegative, nondecreasing functions $D_i$, $K_i$ and $R_i$ represent the fluid analogs of  cumulative class $i$ 
departures from service, cumulative entries to service and cumulative reneging,  and
let $\nu_i$ represent the fluid analog of the
measure-valued function that encodes the ages of class $i$ jobs in service.
Since we assume exponential reneging times, we will not require the potential reneging measures $\eta_i$, but only the reneging rate $\theta_i > 0$.   
Also, let $(X,\theta, \nu,B,Q,D,K,R)$ be the corresponding vector-valued processes whose $i$th component is given by
	$(X_i,\theta_i, \nu_i,B_i,Q_i,D_i,K_i,R_i)$.  We describe the fluid  equations only for the  special case when all
	service distributions are identical, with common cumulative distribution  function $G = G^s$, 
	hazard rate function $h = h^s$ and  support $[0,H) = [0,H^s)=[0,\iy)$, and
	arrival rates $\lambda_i  > 0, i = 1, \ldots, J$.

Before we present the fluid model equations, let us comment on the  special form that
the single-server fluid model equations
(of Definition \ref{def-fleqns})
take when the reneging is exponential. In this case, the reneging hazard rate is constant,
namely $h^r(t)=\theta$ for all $t$, and thus equation \eqref{fr} takes the form 
\[
R(t)=\theta\int_0^tQ(s)ds, 
\]
and thus there is no longer any need to keep track of the potential reneging measure
 $\eta$. 
Accordingly, in the multiclass setting, our fluid model is an extension
of such a modified set of fluid equations where the equation of $R$ is similar to
the above display, and from which $\eta$ is absent.

	\begin{definition}
		\label{def-fluidmulti}
		Given arrival and reneging rate vectors $\lambda \in (0,\infty)^J$ and $\theta \in (0,\infty)^J$, and 
		initial   condition $(X(0), \nu_0) \in [0,\infty)^J \times (\calM_F[0,\iy))^J$,
		    a tuple
                    $(B,X,Q,D,K,R,\nu)
                    \in (\calD_{\R^J_+}(\R_+))^3 \times  (\calD^+_{\R^J_+}(\R_+))^3 \times (\calD_{\calM_F[0,\iy)}(\R_+))^J$  
		is said to be a {\it solution to the multiclass fluid equations 
			with initial condition $(X(0),\nu_0)$ and arrival and reneging rate vectors $\lambda$ and $\theta$} if 
		equations \eqref{03+}--\eqref{04+} below are satisfied: 
		For $\ph\in {\mathcal C}^1_c([0,\iy)\times\R_+)$, and $t \geq 0$, 
		\begin{align}
			\notag
			\mean{\ph(\cdot,t),\nu_{i,t}} &= \mean{\ph(\cdot,0),\nu_{i,0}}
			+\int_0^t\mean{\ph_x(\cdot,s)+\ph_t(\cdot,s),\nu_{i,s}}ds\\
			&\quad
			\label{03+}
			-\int_0^t\mean{h(\cdot)\ph(\cdot,s),\nu_{i,s}}ds+\int_0^t\ph(0,s)dK_{i}(s),
		\end{align}
		where $B, D, R$ are the auxiliary processes  given by  
		\begin{equation}
			\label{04+}
			B_{i}(t) = \mean{1,\nu_{i,t}}, \qquad
			D_{i}(t) = \int_0^t  \langle h, \nu_{i,s} \rangle ds, \qquad
			R_{i}(t) = \theta_i\int_0^tQ_{i}(s)ds.
                        \end{equation}
 and for $t \geq 0$, $K, B, D$ satisfy the following  balance equations and basic relations: 
		\begin{align}
			\label{01+}
			B_i &= B_{i,0}-D_i+K_i,\\
			\label{07+}
			X_i(t) &= X_{i,0}-D_i(t)+\lambda_it-R_i(t),\\
			\label{08+}
			Q_i &= X_i-B_i,
		\end{align}
		as well as  conditions imposing work conservation and non-preemptive priority: 
		\begin{align}
			\label{02+}
			& I:= 1-\sum_{i=1}^JB_i=\Big(1-\sum_{i=1}^JX_i\Big)^+,\\
			\label{10+}
			&K_{i}(t)=\int_{[0,t]}1_{\{\sum_{j=1}^{i-1}Q_{j,s}=0\}}dK_{i}(s),\qquad i\ge2,\, t\ge0. 
		\end{align}  
	\end{definition}

Under the assumption of bounded reneging hazard rates, which is indeed fulfilled
when the reneging distribution is exponential, it was shown in \cite[Theorem 3.1]{AtaKasShi13}
that uniqueness holds for solutions of the fluid  equations
for any given data and initial conditions.
Existence of solutions was also established there by showing that
the scaling limit of the underlying queueing system is a solution.

By the same argument given in the proof of Theorem \ref{th-fluid}, it follows from the
results  in Theorem 4.1 of \cite{KasRam11} that 
the measure-valued age  equation \eqref{03+} implies  that for every $\fsp \in {\mathcal C}_b([0,\iy))$ or $\fsp = h$, 
\begin{equation}\label{60}
	\lan \fsp,\nu_{i,t}\ran=\int_{[0,\iy)}\frac{\bar G(x+t)}{\bar G(x)}
	\fsp (x+t)\nu_{i,0}(dx)
	+\int_{[0,t]}\bar G(s-x)\fsp (s-x)dK_{i}(s),  
\end{equation}
where recall $\bar{G} = 1-G$. 
In what follows, given a vector or vector-valued process $Y$, we use  $\tilde Y$ to be generic notation for the sum $\sum_{i=1}^JY_i$. By \eqref{60},
$\tilde \nu$ and $\tilde K$ satisfy, for every $\fsp \in {\mathcal C}_b([0,\iy))$ or $\fsp = h$, 
\begin{equation}\label{61}
	\lan \fsp,\tilde\nu_{t}\ran=\int_{[0,\iy)}\frac{\bar G(x+t)}{\bar G(x)}
	\fsp(x+t)\tilde\nu_{0}(dx)+\int_{[0,t]}\bar G(t-s)\fsp(t-s)d\tilde K (s). 
\end{equation}
In other words,  \eqref{f5} holds with $(\nu,K)$ and $G^s$ replaced with $(\tilde\nu,\tilde K)$ and $G$.
We now argue that, $\tilde{K}$ and $\tilde{B}$ satisfy the analog of \eqref{fkprime}.
First, note that by \eqref{04+}, $\tilde B=\lan1,\tilde\nu\ran$, and if
$\tilde B_t<1$ then, on an open interval containing $t$ we have $\tilde X<1$ due to  \eqref{02+}. 
Hence, $\tilde Q=0$ by \eqref{08+} and  $\tilde R=0$ by \eqref{04+}. Hence, subtracting
\eqref{07+} from \eqref{01+}, $\tilde K=\tilde E+c$ on this interval (where $c$
does not depend on time), and so   $\tilde K^\prime(s)=\tilde\la$ holds on the interval.
Combining this with 
\[
\tilde K(t) =\tilde B(t)-\tilde B(0)+\int_0^t\lan h,\tilde\nu_s\ran ds, 
\]
which follows from \eqref{01+} and \eqref{04+}, we obtain,
exactly as in \cite[Theorem 3.2]{AtaKasShi13}, that
for a.e.\ $t$, $\tilde{K}^\prime (t)=\tilde k(t)$ where
\begin{equation}\label{110}
	\tilde k (t) =\begin{cases}
		\lan h,\tilde\nu_t\ran, & \tilde B(t) =1,\\
		\tilde\la, & \tilde B(t) <1.
	\end{cases}
\end{equation}

\subsection{Results for the Multiclass System}
\label{subs-multi2} 

We will be interested in the supercritical case 
where $\sum_i\la_i>1$ and $\theta_{\min}=\min_i\theta_i>0$.
Let $\rho_i$, $i=1,\ldots,J$ be characterized by
\[
\sum_{i=1}^j\rho_i=\Big(\sum_{i=1}^j\la_i\Big)\w1,\qquad j=1,\ldots,J,
\]
and let
\[
q_i=\frac{\la_i-\rho_i}{\theta_i},\qquad i=1,\ldots,J.
\]

We now state the main result.

\begin{theorem}\label{th2}
	Suppose that $h$ satisfies Assumption \ref{ass:a}, and 
	$\lambda, \theta \in (0,\infty)^J$ are such that
	$\tilde{\la} = \sum_{i=1}^J\la_i>1$, and 
     $(X_0,\nu_0) \in \R_+^J \times ({\mathcal M}_F[0,\iy))^J$ satisfies 
		$1 - \langle \f1, \tilde{\nu} \rangle = (1 - \tilde{X})^+$. 
		Then  any solution $(B, X, Q, D, K, R,\nu)$ to the multiclass fluid  equations 
		with initial condition  $(X_0,\nu_0)$ and arrival and reneging rate vectors $\lambda$ and  $\theta$ 
		satisfies  $\nu_{i,t} \Rightarrow \rho_i\nu^*$ and 
	$Q_i(t)\to q_i$ as $t \to \iy$ for $i = 1, \ldots, J$.
\end{theorem}

\begin{remark}
	{\em 
	This validates Theorem 5.1 of \cite{AtaKasShi13} in the special case where for all $i$, 
	$h^s_i=h$, with  $h$ satisfying Assumption \ref{ass:a}. }
\end{remark}

\begin{remark}
	\label{rem-multi}
	{\em  
  The characterizations in \eqref{60} and \eqref{110}  show that the  aggregate
			processes $(\tilde X, \tilde\nu)$ and $(\tilde D,\tilde K, \tilde R, \tilde  S, \tilde Q, \tilde B)$ satisfy the fluid equations of the single class case (see  Definition \ref{def-fleqns}), subject to the simplification described at the beginning of Section \ref{subs-multi1},
			where in particular reneging is given directly by \eqref{04+} and
			the process $\eta$ is not used. 
			Hence, in the supercritical setting
			$\tilde\la>1$,
			we may conclude from Theorem \ref{th-main2}(2) that, with $\nu^*(dx)=\bar G(x)dx$,
			one has $\tilde\nu_t\Rightarrow \nu^*$ and that there exists
			$T < \infty$ such that $\tilde B_t=1$ for all $t \geq T$.
			Moreover, by Proposition \ref{prop-superpf}, $\lan h,\tilde\nu_t\ran\to1$.
			As a result, by \eqref{fkprime}, one has $\tilde k_t=\lan h,\tilde\nu_t\ran$
			for all large $t$, and hence also $\tilde k_t\to1$.
                        }
	\end{remark}

\begin{proof}[Proof of Theorem \ref{th2}]
	In this proof, the special case in which there exists $ i_0  \in \{1, \ldots, J-1\} $ such that $\sum_{i=1}^{i_0}\la_i=1$
	is called the {\it borderline case}, and the more typical case,
	where such $i_0$ does not exist, is called the {\it typical case}.
	
	If $\la_1<1$, set $\ell := \max\{j:\sum_{i=1}^j\la_i<1\}$,
	otherwise let $\ell=0$. Also, set $m=\ell+1$.
	Then, since by assumption $\tilde{\la} > 1$,   by the definition of $m$,
	we have  $\sum_{i=1}^m\la_i=1$ [respectively, $>1$] in the borderline case
	[respectively, in the typical case].
     Also, in what follows,  we use the hat (when $\ell \geq 1$) and $\#$ notation for summation up to $l$ and, respectively, $m$, as in
		\begin{equation}
			\label{hat-notation}
			\hat Y=\sum_{i=1}^\ell Y_i,  \quad \mbox{  and } \quad   Y^{\#} =\sum_{i=1}^m Y_i, \qquad   Y= \lambda, X, \nu, D, K, R, B. 
		\end{equation}
		(in addition to the notation already introduced, $\tilde Y=\sum_{i=1}^J Y_i$).
	
	The structure of the proof is as follows. In Step 1 we prove the assertions for $i\le\ell$.
	Steps 2 and 3 address the remaining classes $i\ge m$ in the typical and borderline cases, respectively. 
	First, note that since $\tilde \la >1$, by Remark \ref{rem-multi},  there exists $T < \infty$ such  that 
	\begin{equation}
		\label{dep-limit}
		\tilde k (t) \to 1,   \lan h,\tilde\nu_t\ran\to1  \mbox{ as  } t \ra \iy \quad \mbox{ and } \quad  \tilde B(t) = 1 \mbox{ for all } t \geq T. 
	\end{equation}
	
	\noindent 
	{\bf Step 1.} Consider the case $\ell\ge1$ (that is, $\la_1<1$).
	In this step we consider classes  $1 \leq i \leq  \ell$ and  establish the claim that there exists $t_1 < \infty$ such that 
		$Q_i(t)=0$ for all  $t \geq t_1$, and moreover,  that $\nu_i(t)\to \rho_i\nu^*$ as $t \ra \iy$. 
	(Note that for $i\le\ell$, $\la_i=\rho_i$, hence the asserted convergence
	$Q_i(t)\to q_i=0$ would then  follow).

	 Recalling the notational convention \eqref{hat-notation}, by the definition  of $\ell$,
		$\hat\la  = \sum_{i=1}^\ell  \lambda_i <1$, and so
	there exist $\eps_0>0$ and $ 0 < t_0 <  \infty$ such that
	$\lan h,\tilde\nu_t\ran>\hat\la+\eps_0$ for all  $t\ge t_0$.
		If $\hat Q(t)= 0$ for all $t \geq t_0$ then the claim follows trivially.  So, we now consider the converse case,  when
		${\mathcal O} := \{t > t_0: \hat Q (t) > 0\}$ is non-empty. 
		Since $\hat Q$ is continuous, ${\mathcal O}$ is open and is a union of countable open intervals. For a.e.\ $s$ in each such interval, 
		by \eqref{10+}, for  all $i>\ell$, 
		$K_i^\prime(s) = 0$. Moreover, since \eqref{08+} and \eqref{02+} together 
		show that, $\tilde Q(t)>0$ implies $\tilde B(t)=1$  for any $t> 0$, we conclude in particular that 
		$\tilde B(s) = 1$. In turn,  \eqref{01+}, \eqref{08+}, \eqref{04+} and the fact that $\tilde{R}$ is non-decreasing
		together imply that for a.e.\ $s \in {\mathcal O}$,  
		$\tilde D^\prime(s) =\tilde K^\prime (s) = \lan h,\tilde\nu_s\ran$. Thus, we have for a.e. $s > t_0$, 
		\begin{align*}
		 \hat	Q(s) >  0 \quad \Rightarrow \quad \hat Q^\prime (s)=\hat\la-\hat R^\prime (s)-\hat K^\prime (s)\le \hat\la -\lan h,\tilde\nu_s\ran \le-\eps_0. 
		\end{align*}
		Thus,  there must exist a finite time, $t_1\ge t_0$, when $\hat Q(t_1) = 0$. Since the last display continues  to 
		hold for all $s \geq t_1$,  applying Lemma \ref{lem:f} with $f = -Q$, $T= t_1$, $S= \infty$   and $c=0$, it follows that 
		for all $t\geq t_1$, $\hat{Q}(t) = 0$, or equivalently,  $Q_i(t)=0$ for all $i\le\ell$.   
	
	 To finish proving  the  claim in Step 1, it only remains to show that $\nu_i(t)\to\rho_i \nu^*=\la_i\nu^*$ for $i\le\ell$.
	 For $t \geq t_1$, it follows from \eqref{08+} and \eqref{04+}, respectively, that for $i\le\ell$,  $X_i(t)=B_i(t)$ and 
		$R_i^\prime(t) = 0$. 
		Hence by \eqref{01+}--\eqref{07+},
		$K_i^\prime(t)=\la_i$ for such $i$ and $t$.
		Substituting these relations in \eqref{60} and taking the large $t$ limit yields (exactly as in the proof of Lemma \ref{lem-reneg}),  
		the convergence of $\nu_i(dx)$ to $\la_i\bar G(x)dx$ as asserted.
	
	\noindent 
	{\bf Step 2.}
	In this step we treat the typical case, proving
	the claim for all the remaining classes $i\ge m=\ell+1$ (where possibly $\ell=0$, $m=1$).
	 Recall the notation in \eqref{hat-notation} and note that   in this case one has
		$\lambda^{\#}=\sum_{i=1}^m\la_i>1$. By the definition of $m$ and $\rho_i$, this implies
	$\rho_m<\la_m$.

       Let $t_1 <  \infty$ be as in Step 1, and assume without loss of generality
		that $t_1 \geq T$, where $T$ is as in  \eqref{dep-limit}.  Then  given $\eps  \in (0,1-\lambda^{\#})$, there exists 
		$t_2 = t_2(\eps) \ge t_1$  such that for all $t\ge t_2$, $|\lan h,\tilde\nu_t\ran-1|<\eps,$  $|\tilde{k}(t) - 1| < \eps$ and $\tilde B(t)=1$.  
		Then \eqref{04+}  implies that $\tilde D_t^\prime (t) =\lan h,\tilde\nu_t\ran \le (1+\eps)$, and 
		since clearly, $(D^\#)' \le \tilde D^\prime$, on $[t_2,\iy)$, we have for all   $\eps_1 \leq   \lambda^{\#} - 1 - \eps$, 
		\[
		\frac{dX^\#}{dt} =\la^\#- \frac{dD^\#}{dt} - \frac{dR^\#}{dt} 
		\ge\la^\#-(1+\eps)-\theta_mQ_m\ge\eps_1-\theta_mQ_m. 
		\] 
		We now argue by contradiction to prove the claim  that there exists $t_3\ge t_2$ such that $Q_m(t_3)>0$.
		Indeed, assume $Q_m$ vanishes on the whole
		interval $[t_2,\iy)$. 
		Then the last display shows that $X^\#(t)\to\iy$, and hence by \eqref{08+} and the fact that \eqref{02+} implies $\tilde{B}$ lies in
		$[0,1]$, 
		$Q^\#(t)\to\iy$. 
		But since $\hat{Q}$ vanishes on $(t_1,\infty) \supset (t_2,\infty)$  by Step 1, this implies
		$Q_m (t) = Q^\#(t) - \hat{Q}(t) = Q^\#(t) \to\iy$, which contradicts the assumption that $Q_m$ is identically zero on
		$[t_2,\iy)$.   This proves the claim.

		Let  $t_3\ge t_2$ be such that $Q_m(t_3)>0$, and let ${\mathcal O}_m  := \{s \in [t_3, \infty): Q_m(s) > 0\}$. 
		We show below that ${\mathcal O}_m=  [t_3,\infty)$. 
		Towards this goal, we will find it more convenient to work
	with the balance equation for $Q^\#$ than with $X^\#$.
	That is, using \eqref{01+}--\eqref{08+} and \eqref{04+}, note that
	\begin{equation}\label{111}
		Q_i^\prime =\la_i-K_i^\prime -\theta_iQ_i, \quad i = 1, \ldots, J,  \qquad \mbox{ and } \qquad     \frac{dQ^\#}{dt}=\la^\#-\frac{dK^\#}{dt}- \sum_{i=1}^m\theta_i Q_i^\#. 
	\end{equation}
	On any open interval in ${\mathcal O}_m$,  $Q_m>0$ and $\hat{Q}=0$, and hence 
	the priority rule \eqref{10+} implies 
 $dK^\#/dt=d\tilde K/dt = \tilde k$, where recall 
		$|\tilde  k(t) - 1| <\eps$. Thus, for  all $t \geq t_3$,   we have 
		\[
		Q_m(t) >0 \quad  \Ra \quad Q_m^\prime(t)=\frac{dQ^\#}{dt}(t) \ge\la^\#-(1+\eps)-\frac{dR^\#}{dt}(t)\ge\eps_1-\theta_mQ_m(t). 
		\]
		Since this  is strictly greater than $\eps_1/2$ whenever $Q_m(t) < \eps_1/2\theta_m$, 
		this clearly implies $Q_m(t) > 0$ for all $t \in  [t_3,\iy)$, as claimed.
		In turn, by the priority rule \eqref{10+}, this implies that on $[t_3,\iy)$, $K_i^\prime=0$ for all $i>m$, and therefore by \eqref{01+}
		and \eqref{04+}, $B_i^\prime = - \langle h, \nu_i\rangle \leq  - \eps_h B_i$, where recall that $\eps_h$ is the strictly
		positive lower
		bound on $h$.
		This shows that for $i>m$,  $B_i (t)\ra 0$ as $t \ra\infty$ and hence, 
		$\nu_i\Rightarrow 0$.
		As $t \to \iy$, since  we already have  convergence  of the aggregate  $\tilde\nu_t\Rightarrow \nu^*$ (see Remark \ref{rem-multi})
		and  $\nu_{i,t}\Rightarrow \la_i\nu^*$ for all $i<m$ (by Step 1),  we conclude
		that $\nu_{m,t}\Rightarrow\rho_m\nu^*$.

       To complete Step 2, it only remains to address the convergence of $Q_i$, $i\ge m$.
		Since, as argued above, for  $t \in [t_3,\iy)$,   $K^\prime_i(t)=0$ for $i>m$, \eqref{111} shows that  $Q_i(t)\to \la_i/\theta_i=q_i$ as $t \to \iy$. 
		As for $Q_m$, note that since  on $[t_3,\iy)$, for $1 \leq i \leq \ell =  m-1$, $Q_i = 0$ by Step 1,  \eqref{111} shows that   $K_i^\prime = \la_i$, or equivalently,
		$\hat{K}^\prime = \hat{\la}$.  Thus, 
		denoting $e(t):=\tilde k(t)-1$, we have $e(t)\to0$, and  recalling that $\rho_m = \hat{\lambda} - 1$,
		\[
		K_m^\prime(t) =\tilde K^\prime (t) -\hat{K}^\prime_i(t) =\tilde k(t) -\hat{\la}
		=(\rho_m+e(t)). 
		\]
		Thus, we obtain
		\[
		Q_m^\prime(t)=(\la_m-\rho_m-e(t))-\theta_mQ_m(t).
		\]
		This implies that as $t \ra \infty$, $Q_m(t)$ converges to $q_m= (\la_m-\rho_m)/\theta_m$. 
		Here, we used the elementary fact that for a differentiable function $u$ on $[0,\infty)$, 
		\begin{equation}
			\label{DE}
			u^\prime(t) =w (t) -\theta u(t) \mbox{ and } u(0) = u_0 \qquad   \Rightarrow  \qquad u(t)=\int_0^te^{-\theta(t-s)}w(s)ds+u_0e^{-\theta t},
		\end{equation}
		which converges to $c/\theta$ whenever $w(t)\to c$ as $t\to\iy$.

	\noindent 
	{\bf Step 3.}
	Lastly, we consider the borderline case, and establish the assertions regarding
	the remaining classes $i\in \{m, \ldots, J\}$. In this case $\la^\#=\sum_{i = 1}^m \la_i=1$.
	
	As in \eqref{110}, the priority structure  specified by \eqref{10+} dictates that $dK^\#/dt=k^\#$, 
		where $k^\#(t)$  is given by $\lan h,\tilde\nu_t\ran$ when $B^\# (t)=1$
		and equal to $\la^\#$ when $B^\# (t)<1$.  Since $\la^\#=1$ and by \eqref{dep-limit}, $\lan h,\tilde\nu_t\ran\to1$ 
	we infer that $k^\# (t)\to1$ as $t\to\iy$.   Summing \eqref{60} over $i\le m$,
	using $\int_{0}^\infty\bar G(x)dx=1$ and applying the test function $\fsp = \f1$ shows
	that $B^\#_t\to1$ as $t\to\iy$ (where the application of bounded continuous $\fsp$ can be justified in the usual 
	manner). Applying general compactly supported
	test functions gives $\nu^\#_t\Rightarrow\nu^*$, where $\nu^*(dx) = \bar{G}(x)dx$. 
	Given the convergence already established for $\nu_{i,t}$, $i\le\ell =m-1$, the convergence
	of $\nu^\#_t$ yields
	that of $\nu_{m,t}\to\la_m\nu^*$ (note that in the borderline case currently considered,
	$\la_m=\rho_m$).   Moreover, the fact that $B^\# (t) \to1$ implies that $\sum_{i=m+1}^JB_{i}(t)=\tilde B(t)-B^\#(t)\to0$,
	and hence, for all $i>m$, $B_{i}(t)\to0$ and consequently $\nu_{i,t}\Rightarrow 0$. 
	
	Next we show that $Q_i(t)\to q_i$ for $i>m$, for which we again use \eqref{111}.
	Combining the convergence  $k^\#(t)\to1$ that we just showed with  $\tilde k(t)\to 1$ from \eqref{dep-limit}, 
	it  follows that  $k_{i}(t)\to0$ for all $i>m$. Recalling that $K_i^\prime =k_i$ and using the first equation  in \eqref{111}
	and \eqref{DE} yields  $Q_i(t)\to\la_i/\theta_i=q_i$, for $i>m$.
	
	We finally show that $Q_{m}(t)\to0$. To this end, note that by the aggregate equation  in \eqref{111}
	and the property that for sufficiently large $t$, $Q_i(t)=0$, for $i\le\ell$, (from Step 1) 
	giving $dR^\#(t)/dt=\theta_mQ_m(t)=\theta_m Q^\#(t)$. Since $\la^\#=1$,  \eqref{111} shows that the  following is valid
	for all large $t$,
	\[
	\frac{dQ^\#}{dt} =1-k^\# (t) -\theta_mQ^\#(t). 
	\]
	Recalling that $k^\# (t) \to1$, and again using \eqref{DE}, 
	it follows that $Q^\#(t)\to0$. Consequently, $Q_{m}(t)\to0$.
	This completes the proof.
\end{proof}

\begin{appendix}

\section{Proof of Lemma \ref{lem1}.}
\label{sec-lem1}
\beginsec
\manualnames{A}

    \begin{proof}[Proof of Lemma \ref{lem1}]
       Fix a  measurable function $f:[0,\infty) \mapsto \R_+$ with $\int_0^\infty f dx \leq 1$. 
       For notational conciseness, define 
\begin{equation*}
   Af := 
   \int_0^\iy h^s(x) f(x)\log\frac{f(x)}{f^*(x)}dx-z_f \log z_f, 
\end{equation*}
where recall $z_f:=\int_0^\iy h^s f dx<\iy$. 
        Let  $U(x):=x\log x$, $x > 0$, $U(0)  = 0$.   Fix a non-negative  measurable function  $\psi$ on $[0,\infty)$ with
        $c_\psi = \int_0^\infty \psi dx \leq 1$. Then, note from the definition of $A$ in \eqref{estimate}  that 
 \begin{equation}
   \label{temp-A}
A(\psi)
= \int_0^\infty U\left( \frac{\psi}{f^*} \right)h^sf^*dx  - U \left( z_\psi\right). 
 \end{equation}
Since $\int_0^\infty h^sf^*dx=1$ and  $\int_0^\infty \left( \frac{\psi}{f^*} \right)h^sf^*dx  = z_\psi$,  
 the convexity of $U$ and Jensen's inequality imply 
 the nonnegativity of $A(\psi)$.  To obtain the more refined estimate \eqref{estimate}, define  
\[
V(x):= U(x)-[U'(z_f)(x-z_f)+U(z_f)].
\]
Then, the strict convexity of $U$ implies $V(x)\ge0$ and $V(x)=0$ if and only if $x=z_f$.  Using \eqref{temp-A}
we have 
\[
  A(\psi)=\int_0^\infty V\Big(\frac{\psi}{f^*}\Big)h^sf^*dx + \int_0^\infty U'(z_f)\Big(\frac{\psi}{f^*}-z_f\Big)\Big]h^sf^*dx
=\int_0^\infty V\Big(\frac{\psi}{f^*}\Big)h^sf^*dx,
\]
where the last equality uses the  definition of $z_\psi$. 
Since $V  \geq 0$,  denoting $c_\psi  := \int_0^\infty \psi dx\le1$, and recalling the functional
$R$ from \eqref{eq-Rfnal}, we have 
\begin{align*}
A(\psi) &\ge \eps_h\int_0^\infty V\Big(\frac{\psi}{f^*}\Big)f^*dx
=\eps_h\Big[\int_0^\infty \psi \log\frac{\psi}{f^*}dx-\int_0^\infty U'(z_\psi)\Big(\frac{\psi}{f^*}-z_\psi\Big)f^*dx-U(z_\psi)\Big]\\
&=\eps_h[R(\mu^\psi\|\nu_*)-c_\psi U'(z_\psi)+U'(z_\psi)z_\psi-U(z_\psi)]\\
&=\eps_h[R(\mu^\psi\|\nu_*)-c_\psi \log z_\psi -c_\psi+z_\psi]\\
&=\eps_h\{R(\mu^\psi\|\nu_*)+c_\psi  [-\log z_\psi -1+z_\psi]+z_\psi(1-c_\psi)\}\\
&\ge\eps_hR(\mu^\psi\|\nu_*),
\end{align*}
where   the third  equality used the fact that $U'(x) = \log x + 1$   and  $U'(x)  x  - U(x) = x$, and  the
last inequality uses the  elementary inequality $x  - \log x \geq 1$ for  all $x > 0$. This proves \eqref{estimate}.  
    \end{proof}

\section{An Elementary Property of Absolutely Continuous Functions}
\label{sec-ac}
\beginsec
\manualnames{B}

    The following simple property is used in Section \ref{subs-pf2supcrit}. 

\begin{lemma} \label{lem:f}
	Let $f$ be an absolutely continuous function defined on $[0,S)$ for some $0<S\leq \infty$. Suppose that there exist a time $T\in (0,S)$ and a constant $c>0$ such that $f(T)\geq c$ and for a.e.\ $t\in (T,S)$, $f^\prime(t)\geq 0$ if $f(t)<c$. Then $f(t)\geq c$ for all $t\in [T,S)$. 
\end{lemma}
\begin{proof}
	Suppose the conclusion of the lemma does not hold. 
	Then there must exist  $ T< t_1 < t_2 $ for which 
	$f(t_1) \geq c$ and $f(t_2) < c$.
	Since $f$ is absolutely continuous,
	there must exist some interval $(s_1,s_2) \subset (t_1,t_2)$
	such that $f(s) < c$  for $s \in (s_1,s_2)$ and $f^\prime (s) < 0$ for $s$ in a  subset
	${\mathcal S} \subset (s_1,s_2)$ of positive Lebesgue measure. 
	However, this contradicts the assumption of the lemma, that is, for a.e.\ $t\in [T,\infty)$, $f^\prime(t)\geq 0$ if $f(t)<c$. Hence the lemma is proved. 
\end{proof}

    \end{appendix}

\noi{\bf Acknowledgement.}
We would like to thank Amber Puha for raising the question of a full justification
of the convergence results in Theorem 3.3 of \cite{KanRam12}.

\footnotesize

\bibliographystyle{is-abbrv}
%%\bibliographystyle{is-alpha}

%\bibliography{main}

\end{document}